\documentclass[11pt]{article}
\usepackage[a4paper,margin=1in]{geometry}
\usepackage{amsmath,amssymb,amsthm,mathtools}
\usepackage{mathrsfs}
\usepackage{microtype}
\usepackage{hyperref}
\usepackage[nameinlink,capitalize]{cleveref}
\usepackage{authblk}
\hypersetup{
    hidelinks,
    pdftitle={A cyclotomic gluing construction for lattice sphere packings},
}
\newtheorem{theorem}{Theorem}[section]
\newtheorem{proposition}[theorem]{Proposition}
\newtheorem{lemma}[theorem]{Lemma}
\newtheorem{corollary}[theorem]{Corollary}
\newtheorem{remark}[theorem]{Remark}
\newtheorem{definition}[theorem]{Definition}
\usepackage{xcolor}
\theoremstyle{definition}
\newtheorem{example}{Example}[section]

\usepackage{tikz}
\usepackage{xcolor}

\newcommand{\vol}{\operatorname{vol}}
\newcommand{\covol}{\operatorname{covol}}

\newcommand{\Ent}{\operatorname{Ent}}

\newcommand{\Res}{\operatorname{Res}}

\title{A cyclotomic gluing construction for lattice sphere packings}
\author[1]{Nihar Gargava}
\author[2]{Matthias Gröbner}
\author[3]{Michael Simkin}
\affil[1]{Institut de Mathématiques d’Orsay, Université Paris-Saclay, France\par
\texttt{nihar.gargava@universite-paris-saclay.fr}}
\affil[2]{Chair of Number Theory, EPFL, Lausanne, Switzerland\par
\texttt{matthias.grobner@epfl.ch}}
\affil[3]{Independent researcher, Tbilisi, Georgia \par
\texttt{simsim314@gmail.com}}
\date{}

\begin{document}
\maketitle

\begin{abstract}
    For a positive integer $d$, let $\Delta_d^L$ denote the supremum of
    the densities of lattice sphere packings in $\mathbb R^d$. We prove
    that, for every sufficiently large $d$,
    $$
    \Delta_d^L
    \gg
    d^2\log\log d\,2^{-d}.
    $$
    Abuya--Gargava--Zhao obtained the same bound along a sparse sequence of
dimensions by combining Venkatesh's cyclotomic symmetries with Klartag's
stochastic ellipsoid method. We introduce a multi-cyclotomic gluing
construction that allows this argument to work in every sufficiently
large dimension. 

\end{abstract}


\section{Introduction}\label{sec:introduction}

For a positive integer \(d\), let \(\Delta_d^L\) denote the supremum of the densities of sphere packings in \(\mathbb R^d\) whose centers form a lattice. Our main result gives an unbounded multiplicative improvement over the best known general lower bound for \(\Delta_d^L\), valid in every sufficiently large dimension. 
The resulting bound also improves the best known general bound for sphere packings not necessarily arising from lattices.


\subsection{Previous results}

The sphere packing problem has a long history. The following is a non-exhaustive list of results that are relevant for this paper.

\begin{enumerate}
    \item 
        Venkatesh exploited cyclotomic symmetries in lattice packing constructions~\cite{Venkatesh}, obtaining $\Delta_{d}^{L} \gg d^{} \log \log d \cdot 2^{-d}$ for infinitely many $d$. 

        The argument applies in dimensions of the form
        $
        d=\varphi(m)r,
        $
        where $\varphi$ is Euler's totient function, $m$ is the order
        of the cyclic action and $r=2$. The gain comes from the
        fact that vectors in a full cyclic orbit impose only one constraint on an
        invariant quadratic form.  
        The ratio
        $m /\varphi(m)$ can be of order $\log\log m$. Hence, the set of dimensions $d \geq 1$ for which this argument works does not even have positive density. The same construction with $r=1$ was developed in~\cite{GargavaViazovska}.
        This idea was studied further in
        \cite{GargavaSerbanViazovska,GargavaModuleShortest}.

    \item 
        Klartag introduced a different mechanism \cite{Klartag}. Starting from a
        lattice-free ball, he uses a stochastic deformation to produce a larger lattice-free ellipsoid, leading to the bound $\Delta_d^L\gg d^2 2^{-d}$ as $d \rightarrow  \infty$.  
        
    \item 
    Abuya-Gargava-Zhao  combined Klartag's stochastic method with the cyclotomic symmetries in Venkatesh's construction~\cite{AbuyaGargavaZhao}. They obtained
        $
        \Delta_d^L
        \gg
        d^2\log\log d\,2^{-d}
        $
        for an infinite sequence of dimensions $d$ of the form $d= \varphi(m) r$. Their argument requires both a choice of $m$ with $m/\varphi(m) \asymp \log\log m$ and a sufficiently large multiplicity $r$. Consequently, the argument applies only to a restricted set of dimensions.  
\end{enumerate}

\subsection{Main theorem}
Our main result extends the logarithmic gain from cyclic
symmetry to every sufficiently large dimension. The proof retains the stochastic-symmetry mechanism of Abuya--Gargava--Zhao, but replaces the single cyclotomic component by a suitably glued collection of cyclotomic components.  

\begin{theorem}\label{thm:main}
    There exist an absolute constant $c>0$ and an integer $d_0\ge1$ such that,
    for every integer $d\ge d_0$, we have 
    $$
    \Delta_d^L
    \ge
    c\,d^2\log\log d\,2^{-d}.
    $$
\end{theorem}

\subsection{Proof strategy}

We begin with the basic idea of gluing. Suppose that $\Lambda_{1} \subseteq \mathbb{R}^{n_{1}}$ and
$\Lambda_{2} \subseteq \mathbb{R}^{n_{2}}$ are two lattices. By a gluing of $\Lambda_{1}$ and $\Lambda_{2}$ we mean a finite-index
sublattice $\Lambda \subseteq \Lambda_{1} \oplus \Lambda_{2}$ whose
projections onto the two coordinate spaces are surjective. The simplest example is the checkerboard
sublattice
$$
\Lambda_{\mathrm{chk}}
=\{(x,y)\in\mathbb Z^2:x\equiv y\pmod 2\}
\subseteq \mathbb Z\oplus\mathbb Z.
$$
Although each coordinate projection is still all of $\mathbb Z$, the two coordinates are no longer independent. They must have the same parity. The highlighted points in~\cref{fig:checkerboard-gluing} illustrate
this elementary mechanism. 

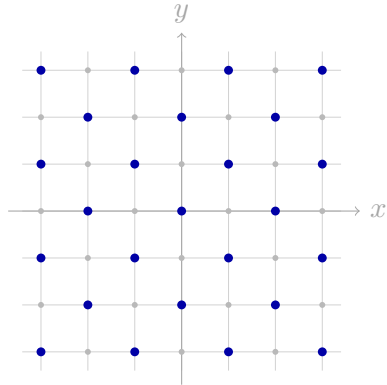
\begin{figure}[htbp]
    \centering
    \begin{tikzpicture}[scale=0.62]
        \draw[step=1,gray!35,very thin] (-3.4,-3.4) grid (3.4,3.4);
        \draw[->,gray!70] (-3.7,0)--(3.8,0) node[right] {$x$};
        \draw[->,gray!70] (0,-3.7)--(0,3.8) node[above] {$y$};
        \foreach \x in {-3,...,3} {
            \foreach \y in {-3,...,3} {
                \pgfmathtruncatemacro{\parity}{mod(\x+\y,2)}
                \ifnum\parity=0
                    \fill[blue!65!black] (\x,\y) circle (2.7pt);
                \else
                    \fill[gray!55] (\x,\y) circle (1.8pt);
                \fi
            }
        }
    \end{tikzpicture}
    \caption{The checkerboard sublattice $\Lambda_{\mathrm{chk}}$ (highlighted points) inside $\mathbb Z^2$.}
    \label{fig:checkerboard-gluing}
\end{figure}

Our construction uses the same principle on several cyclotomic components.  A naive direct product of cyclotomic components is
not sufficient for our purposes. It contains many vectors supported on only one or a few components, which may have smaller cyclic orbits and thereby offset the gain coming from the symmetry. The gluing forces vectors supported on a proper subset of the components to satisfy additional divisibility conditions. These conditions are quantified by cyclotomic resultants, which determine the covolumes of the corresponding support lattices.

The gluing is combined with a volume-preserving rescaling of the individual
cyclotomic components. For each support $S$, the geometric effect
of restricting to $S$, the arithmetic penalty coming from gluing, and the
effect of the rescaling are collected into a single support exponent. This quantity controls the corresponding Haar-averaged short-vector contribution. The main task is therefore to arrange that this exponent is uniformly negative on every proper
support. At the same time, we need the order \(M\) of the cyclic symmetry to be large compared with the rank \(C\) of the cyclotomic order. The ratio \(M/C\) measures the symmetry gain relative to the dimensional cost of the construction, and the stochastic argument ultimately produces a packing density of order
$
2^{-d} d^2 M/C.
$
We therefore seek cyclotomic data for which all proper supports are uniformly suppressed while $M/C \gg\log\log d.$ To achieve this, we encode the resultant interactions by a weighted graph and the conditions are reduced to a combinatorial
problem on this graph.  The particular sets of cyclotomic indices used in our construction are assembled from divisor cubes.

The inputs from gluing, rescaling and Haar averaging are enough to obtain the bound
$\Delta_d^L\gtrsim d\log\log d\,2^{-d}$ in every sufficiently large dimension; see Appendix \ref{app:deterministic-packing}. To obtain the additional factor of $d$, we incorporate Klartag's stochastically evolving ellipsoid. As in Abuya--Gargava--Zhao~\cite{AbuyaGargavaZhao}, the stochastic
deformation is restricted to quadratic forms invariant under the cyclic
symmetry of the construction. Full-support vectors have large cyclic
orbits, so an entire orbit imposes only one constraint on the invariant
deformation. Proper-support vectors may have smaller orbits, but the gluing
estimates above ensure that their total contribution is negligible. Under the same type of multiplicity hypothesis as in~\cite{AbuyaGargavaZhao}, the stochastic argument then yields the bound of \cref{thm:main}.

\subsection{A toy example illustrating the gluing mechanism}
\label{subsec:toy-gluing}

We illustrate the gluing mechanism with a toy example built from two ordinary
ensembles of random lattices. It shows why their direct product has too many
vectors on proper supports, and how a finite-index gluing suppresses them.
Let $r\ge2$, set
$
V=\mathbb R^r\oplus\mathbb R^r
$
and consider first the direct product lattice
$
\Lambda_0=\mathbb Z^r\oplus\mathbb Z^r.
$
To randomize the two factors independently, let
$
G=
\operatorname{SL}_r(\mathbb R)
\times
\operatorname{SL}_r(\mathbb R)
$
act componentwise on $V$.  Thus
$
(g_1,g_2)\Lambda_0
=
g_1\mathbb Z^r\oplus g_2\mathbb Z^r.
$
Here $g_1\mathbb Z^r$ and $g_2\mathbb Z^r$ are independent random
    unimodular lattices, each distributed according to Haar probability
    measure. Siegel's mean-value formula in the two factors gives, for every nonnegative
bounded compactly supported Borel function
$f:\mathbb R^{2r}\to\mathbb R$,
\begin{equation*}
    \begin{aligned}
        \mathbb E_{g_1,g_2}
        \sum_{(u,v)\in\mathbb Z^r\oplus\mathbb Z^r}
        f(g_1u,g_2v)
        =
        f(0,0)
        +
        \int_{\mathbb R^r}f(x,0),dx
        +
        \int_{\mathbb R^r}f(0,y),dy
        +
        \int_{\mathbb R^{2r}}f(x,y),dx,dy.
    \end{aligned}
    \label{eq:toy-product-mean}
\end{equation*}
The two middle terms are the contributions of vectors supported on only
one of the two components. They are not negligible. For $s\ge1$ and $R>0$, let $B_s(R)$
denote the Euclidean ball of radius $R$ centered at the origin
in $\mathbb R^s$. Taking $f=\mathbf1_{B_{2r}(t)}$ in the preceding
mean-value formula gives
\[
\mathbb E_g|g\Lambda_0\cap B_{2r}(t)|
=
1+V_{2r}(t)+2V_r(t).
\]
To compare the contributions from proper and full supports,
we use the volume formula and Stirling's approximation to obtain
$$
\frac{V_r(t)}{V_{2r}(t)}
=
2^{r/2}
\left(\frac{2}{\pi r}\right)^{1/4}
V_{2r}(t)^{-1/2}(1+o(1)).
$$
Thus, if the ambient volume $V_{2r}(t)$ is only polynomial in $r$,
the terms from proper supports are exponentially larger than the
full-dimensional term.  A direct product of good ensembles of random lattices
therefore does not behave like a good random ensemble in the product
dimension.

The point of gluing is to leave the ambient lattice essentially unchanged
while making vectors confined to a proper support arithmetically expensive. For example, consider the sublattice
\begin{equation*}
    \Lambda_3
    :=
    \bigl \{
        (u,v)\in\mathbb Z^r\oplus\mathbb Z^r:
        u\equiv v\pmod3
    \bigr\}.
    \label{eq:toy-glued-lattice}
\end{equation*}
It has index
$
[\mathbb Z^{2r}:\Lambda_3]=3^r.
$
If the second component vanishes, then the congruence
condition forces the first component to be divisible by $3$. That is,
$
\Lambda_3\cap(\mathbb R^r\oplus\{0\})
=
3\mathbb Z^r\oplus\{0\},
$
and similarly
$
\Lambda_3\cap(\{0\}\oplus\mathbb R^r)
=
\{0\}\oplus3\mathbb Z^r.
$
Normalize the ambient lattice to covolume one by setting
$
\widetilde\Lambda_3
:=
3^{-1/2}\Lambda_3.
$
Then each lattice supported on one coordinate subspace has covolume
$
3^{r/2}.
$
Averaging over the normalized Haar ensemble of glued lattices,
the supportwise Siegel mean-value formula gives 
\begin{equation*}
    \mathbb E_g
    \left|g\widetilde\Lambda_3\cap B_{2r}(t)\right|
    =
    1
    +
    V_{2r}(t)
    +
    2 \times 3^{-r/2}V_r(t).
    \label{eq:toy-glued-mean}
\end{equation*}
Thus the ratio of either proper-support contribution to the full-support contribution is
$$
3^{-r/2}
\frac{V_r(t)}{V_{2r}(t)}
=
\left(\frac23\right)^{r/2}
\left(\frac{2}{\pi r}\right)^{1/4}
V_{2r}(t)^{-1/2}(1+o(1)).
$$
Hence, at polynomial ambient volume, the two contributions from proper supports
are exponentially small. This toy example captures the mechanism used below. Gluing converts the vanishing of some components into divisibility conditions on the remaining ones, increasing the covolume of the corresponding support lattice and thereby suppressing vectors of proper support.

\section{The cyclotomic gluing framework and Haar averaging}
\label{sec:gluing-framework}

We introduce the underlying cyclotomic order and its cyclic action, then compute
the indices associated with prescribed supports. Haar averaging converts these
indices into lattice-point estimates for each support, leading to a criterion that
ensures vectors of proper support are exponentially rare. 

\subsection{Cyclotomic orders and the cyclic action}
\label{subsec:cyclotomic-order}

For $n\ge1$, let $\Phi_n(X)$ denote the $n$-th cyclotomic polynomial and
let $\zeta_n=e^{2\pi i/n}$. Let $c_n := \varphi(n)$. Put
$$
K_n:=\mathbb Q(\zeta_n),
\qquad
\mathcal O_n:=\mathbb Z[\zeta_n].
$$ 
For a finite set $\mathcal N\subset\mathbb N$, define
$$
F(X)=\prod_{n\in\mathcal N}\Phi_n(X),
\qquad
\mathcal A=\mathbb Z[X]/(F).
$$
Since $F$ is monic of degree
$
C=\sum_{n\in\mathcal N}\varphi(n),
$
the ring $\mathcal A$ is a free $\mathbb Z$-module of rank $C$. 
Because the $\Phi_n$ are the pairwise coprime minimal polynomials of the
$\zeta_n$, evaluation at the roots $\zeta_n$ gives an injective homomorphism
$$
\mathcal A
\hookrightarrow
\mathcal O
:=
\bigoplus_{n\in\mathcal N} \mathcal O_n.
$$
Since $\mathcal A$ and $\mathcal O$ are free $\mathbb Z$-modules of the same rank $C$, the image of $\mathcal A$ in $\mathcal O$ has finite index. After tensoring with $\mathbb Q$, the Chinese remainder theorem gives
$$
\mathcal A\otimes_{\mathbb Z}\mathbb Q
\simeq
\bigoplus_{n\in\mathcal N} K_n.
$$
We refer to the summands on the right as the cyclotomic components. This
decomposition generally exists only over $\mathbb Q$. Integrally, $\mathcal A$
need not equal the direct sum $\mathcal O$. Thus membership in $\mathcal A$
imposes arithmetic compatibility conditions between the different cyclotomic
components. We refer to these conditions as gluing, in analogy with the
lattice-theoretic terminology of Conway--Sloane~\cite{ConwaySloane}. For
$v=(v_n)_{n\in\mathcal N}\in\bigoplus_{n\in\mathcal N}K_n$, its support
$\operatorname{supp}(v)$ is the set of indices $n$ for which $v_n\ne0$. We
say that $v$ has full support if $\operatorname{supp}(v)=\mathcal N$, and
proper support if $\operatorname{supp}(v)\subsetneq\mathcal N$. When we discuss
nonzero vectors of proper support, the support is understood to be nonempty.
We will seek choices of $\mathcal N$ for which gluing strongly penalizes
elements supported on only a proper subset of the components. The
corresponding indices, which quantify the strength of this gluing, will be
computed in the next subsection.

\begin{example}[Gluing of two components]
    \label{ex:two-component-gluing}
    Take $\mathcal N=\{1,3\}.$
    Then
    $$
    \mathcal A
    =
    \mathbb Z[X]/\bigl((X-1)\Phi_3(X)\bigr)
    \hookrightarrow
    \mathcal O
    =
    \mathbb Z\oplus\mathbb Z[\zeta_3].
    $$
    Writing $\beta=m+n\zeta_3$, a pair
    $(a,\beta)\in\mathbb Z\oplus\mathbb Z[\zeta_3]$ lies in the image of
    $\mathcal A$ precisely when
    $$
    a\equiv m+n\pmod 3.
    $$
    In particular, if $a=0$, so that the element is supported only on the cyclotomic component corresponding to $n=3$, then
    $
    m+n\equiv0\pmod3.
    $
    Equivalently, the possible values of $\beta$ form the sublattice
    $
    (1-\zeta_3)\mathbb Z[\zeta_3]
    \subset \mathbb Z[\zeta_3],
    $
    which has index $3$, since
    $
    \mathbb Z[\zeta_3]\big/
    (1-\zeta_3)\mathbb Z[\zeta_3]
    \simeq \mathbb F_3.
    $
    The point of the example is that forcing one cyclotomic component to
    vanish imposes a nontrivial restriction on the remaining component.
    Indeed, it forces it into a proper sublattice.  This gluing condition is illustrated in \cref{fig:two-component-gluing}.
\end{example}

\begin{figure}[ht]
    \centering

    \definecolor{resA}{RGB}{64,145,140}
    \definecolor{resB}{RGB}{226,140,67}
    \definecolor{resC}{RGB}{140,110,180}
    \definecolor{planeFill}{RGB}{245,247,250}
    \definecolor{planeEdge}{RGB}{170,176,184}
    \definecolor{forbidden}{RGB}{205,210,216}
    \definecolor{textgray}{RGB}{90,90,90}

    \begin{tikzpicture}[scale=1.0, every node/.style={font=\small}]

        \tikzset{
            allowedA/.style={circle, fill=resA, inner sep=1.55pt},
            allowedB/.style={circle, fill=resB, inner sep=1.55pt},
            allowedC/.style={circle, fill=resC, inner sep=1.55pt},
            forbiddenpt/.style={circle, fill=forbidden, inner sep=0.9pt},
            plane/.style={fill=planeFill, draw=planeEdge, line width=0.4pt},
            subtitle/.style={text=textgray},
        }


        \begin{scope}[xshift=0cm]

            \node[subtitle] at (0,5.25)
                {$\mathcal O=\mathbb Z\oplus\mathbb Z[\zeta_3]$};

            \foreach \a in {0,1,2} {
                \pgfmathsetmacro{\sx}{-0.38*\a}
                \pgfmathsetmacro{\sy}{1.38*\a}

                \begin{scope}[shift={(\sx,\sy)}]

                    \path[plane]
                        (-0.2,-0.2)
                        -- (3.2,-0.2)
                        -- (2.15,1.75)
                        -- (-1.25,1.75)
                        -- cycle;

                    \foreach \m in {0,...,6} {
                        \foreach \n in {0,...,4} {

                            \pgfmathsetmacro{\x}{0.54*\m-0.27*\n}
                            \pgfmathsetmacro{\y}{0.47*\n}
                            \pgfmathtruncatemacro{\r}{mod(\m+\n,3)}

                            \ifnum\r=0
                                \node[allowedA] at (\x,\y) {};
                            \fi
                            \ifnum\r=1
                                \node[allowedB] at (\x,\y) {};
                            \fi
                            \ifnum\r=2
                                \node[allowedC] at (\x,\y) {};
                            \fi
                        }
                    }

                    \node[left,text=textgray] at (-1.25,0.75) {$a=\a$};

                \end{scope}
            }

            \draw[->,thick]
                (-1.95,0.15) -- (-2.65,3.95)
                node[above] {$a\in\mathbb Z$};

            \node[align=center] at (0.9,-0.55)
                {\small every pair $(a,\beta)$ is allowed};

        \end{scope}


        \begin{scope}[xshift=8.4cm]

            \node[subtitle] at (0,5.25)
                {$\mathcal A\subset \mathbb Z\oplus\mathbb Z[\zeta_3]$};

            \foreach \a in {0,1,2} {
                \pgfmathsetmacro{\sx}{-0.38*\a}
                \pgfmathsetmacro{\sy}{1.38*\a}

                \begin{scope}[shift={(\sx,\sy)}]

                    \path[plane]
                        (-0.2,-0.2)
                        -- (3.2,-0.2)
                        -- (2.15,1.75)
                        -- (-1.25,1.75)
                        -- cycle;

                    \foreach \m in {0,...,6} {
                        \foreach \n in {0,...,4} {

                            \pgfmathsetmacro{\x}{0.54*\m-0.27*\n}
                            \pgfmathsetmacro{\y}{0.47*\n}
                            \pgfmathtruncatemacro{\r}{mod(\m+\n,3)}
                            \pgfmathtruncatemacro{\ok}{mod(\m+\n-\a,3)}

                            \ifnum\ok=0
                                \ifnum\r=0
                                    \node[allowedA] at (\x,\y) {};
                                \fi
                                \ifnum\r=1
                                    \node[allowedB] at (\x,\y) {};
                                \fi
                                \ifnum\r=2
                                    \node[allowedC] at (\x,\y) {};
                                \fi
                            \else
                                \node[forbiddenpt] at (\x,\y) {};
                            \fi
                        }
                    }

                    \node[left,text=textgray] at (-1.25,0.75) {$a=\a$};

                \end{scope}
            }

            \draw[->,thick]
                (-1.95,0.15) -- (-2.65,3.95)
                node[above] {$a\in\mathbb Z$};

            \node[align=center] at (1.35,-0.55)
                {\small only $m+n\equiv a\pmod{3}$ is allowed};

        \end{scope}


        \begin{scope}[yshift=-1.5cm,xshift=-1.8cm]

            \node[allowedA] at (0.6,0) {};
            \node[anchor=west] at (0.78,0)
                {\small $m+n\equiv0\pmod{3}$};

            \node[allowedB] at (4.75,0) {};
            \node[anchor=west] at (4.93,0)
                {\small $m+n\equiv1\pmod{3}$};

            \node[allowedC] at (8.9,0) {};
            \node[anchor=west] at (9.08,0)
                {\small $m+n\equiv2\pmod{3}$};

        \end{scope}

    \end{tikzpicture}

    \caption{
        The gluing for the index set $\mathcal N=\{1,3\}$.
        On the left, in the ambient product $\mathcal O$, the two cyclotomic
        components vary independently.
        On the right, the gluing condition leaves only one of the three residue classes
        of the triangular lattice on each layer indexed by $a$.
    }
    \label{fig:two-component-gluing}
\end{figure}
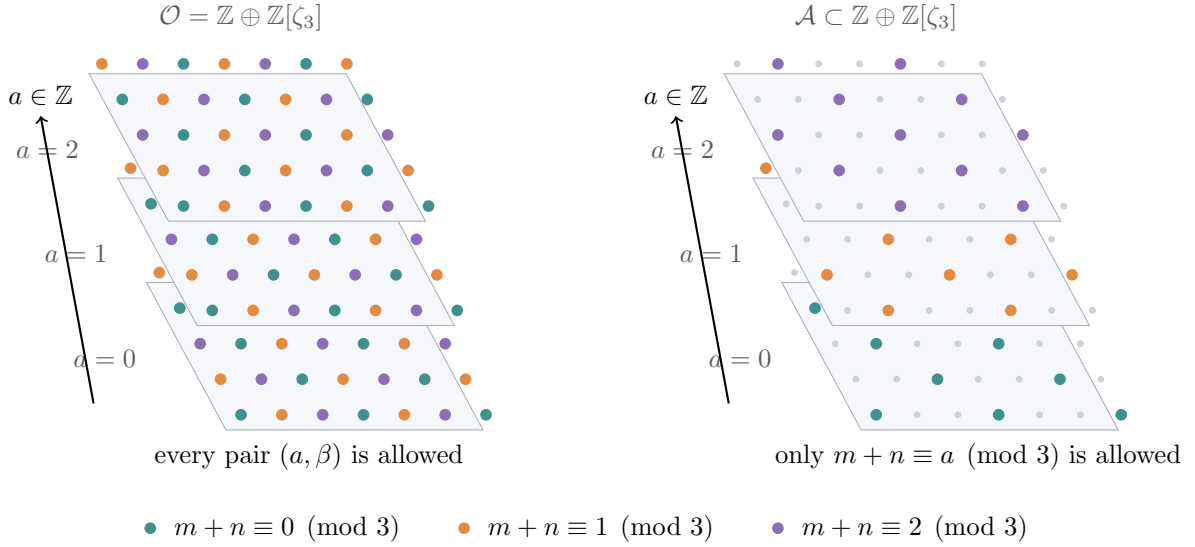

We now describe the cyclic action. Let
$
M=\operatorname{lcm}\{n:n\in\mathcal N\}.
$
For every nonempty $S\subseteq\mathcal N$, write
$$
    M_S=\operatorname{lcm}\{n:n\in S\}.
$$
Then $M_{\mathcal N}=M$.
Since $F(X)\mid X^M-1$, multiplication by $X$ defines an automorphism
$\tau$ of $\mathcal A$ satisfying $\tau^M=\mathrm{Id}$. Thus $\tau$
defines a cyclic action of order $M$ on $\mathcal A$ and on
$\mathcal A\otimes_{\mathbb Z}\mathbb Q$. Under the decomposition
$
\mathcal A\otimes_{\mathbb Z}\mathbb Q
\simeq
\bigoplus_{n\in\mathcal N}K_n,
$
the class of $X$ maps to $(\zeta_n)_{n\in\mathcal N}$.  Consequently,
$
(\tau^jv)_n=\zeta_n^jv_n,
$
and so the restriction of $\tau$ to the $n$-th cyclotomic component has order $n$. For $v\in\mathcal A\otimes_{\mathbb Z}\mathbb Q$, let $ \operatorname{Orb}(v) := \{\tau^jv:j\in\mathbb Z\} $ denote its orbit under the cyclic action.

\begin{lemma} \label{lem:orbit}
    Let $0\ne v\in\mathcal A\otimes_{\mathbb Z}\mathbb Q$ with support
    $S\subset\mathcal N$. Then
    $$
    \#\operatorname{Orb}(v)
    =
    M_S.
    $$

\end{lemma}

\begin{proof}
    For $n\in S$, the $n$-th component is fixed by $\tau^j$ if and only if
    $\zeta_n^j=1$, equivalently $n\mid j$.  Hence $v$ is fixed by
    $\tau^j$ precisely when $j$ is divisible by every $n\in S$.  The
    least positive such $j$ is
    $
    \operatorname{lcm}\{n:n\in S\}.
    $
\end{proof}

Thus every vector of full support has orbit size $M$, while the only vectors that can fail to have cyclic symmetry of order $M$ are vectors of proper support.

\subsection{Cyclotomic resultants and support indices}
\label{subsec:resultants}

We now compute the indices that arise when elements of $\mathcal A$ are
restricted to prescribed sets of cyclotomic components.  These indices
will later determine the covolumes of the corresponding 
sublattices. 
For $S\subset\mathcal N$, write
$$
\mathcal O_S=\bigoplus_{n\in S}\mathbb Z[\zeta_n],
\qquad
K_S=\bigoplus_{n\in S}\mathbb Q(\zeta_n),
\qquad
c_S=\sum_{n\in S} c_n.
$$
For monic polynomials $f,g\in\mathbb Z[X]$, let $\operatorname{Res}(f,g)$ denote their resultant. Put
$$
I_{\mathrm{all}}
=
\prod_{\{m,n\}\subset\mathcal N}
\left|\operatorname{Res}(\Phi_m,\Phi_n)\right|.
$$
For $S\subset\mathcal N$, we separate the pairs involving components in
$S$ into two types. Define 
$$
I_{\mathrm{internal}}(S)
=
\prod_{\{m,n\}\subset S}
\left|\operatorname{Res}(\Phi_m,\Phi_n)\right| , \qquad  
I_{\mathrm{cross}}(S)
=
\prod_{\substack{m\in S\\ n\in \mathcal N\setminus S}}
\left|\operatorname{Res}(\Phi_m,\Phi_n)\right|. 
$$
The following determinant calculation already appears in Mahatab--Sampath~\cite{MahatabSampath}.
\begin{lemma} 
    \label{lem:support-index}
    We have
    $$
    [\mathcal O:\mathcal A]=I_{\mathrm{all}},
    $$
    and, for every $S\subset\mathcal N$,
    $$
    [\mathcal O_S:\mathcal A\cap K_S]
    =
    I_{\mathrm{internal}}(S)I_{\mathrm{cross}}(S).
    $$
\end{lemma}

\begin{proof}
    The first identity follows by computing the determinant of the evaluation map
    $
    \mathbb Z[X]/(F)
    \to
    \bigoplus_{n\in\mathcal N}\mathbb Z[\zeta_n].
    $
    With respect to the monomial basis on the left and the power basis of each
    cyclotomic ring on the right, its determinant is, up to sign,
    $
    \prod_{\{m,n\}\subset\mathcal N}
    \operatorname{Res}(\Phi_m,\Phi_n).
    $
    Taking absolute values therefore gives
    $
    [\mathcal O:\mathcal A]=I_{\mathrm{all}}.
    $
    Now fix $S\subset\mathcal N$ and set
    $$
    F_{S^c}=\prod_{n\notin S}\Phi_n,
    \qquad
    \mathcal A_S=\mathbb Z[X]\Big/\Bigl(\prod_{m\in S}\Phi_m\Bigr),
    $$
    viewing $\mathcal A_S$ as an order in $K_S$ via evaluation.
    An element of $\mathcal A$ lies in $K_S$ precisely when its polynomial
    representative vanishes at every $\zeta_n$ with $n\notin S$, or
    equivalently is divisible by $F_{S^c}$.  Hence
    $
    \mathcal A\cap K_S=F_{S^c}\mathcal A_S.
    $
    By the first identity, we have 
    $
    [\mathcal O_S:\mathcal A_S]=I_{\mathrm{internal}}(S).
    $
    For a number field $K$, write $N_{K/\mathbb Q}$ for its field norm.
    Multiplication by $F_{S^c}$ on the $m$-th component has absolute 
    determinant
    $$
    \left|N_{K_m/\mathbb Q}\bigl(F_{S^c}(\zeta_m)\bigr)\right|
    =
    \prod_{n\notin S}
    \left|N_{K_m/\mathbb Q}\bigl(\Phi_n(\zeta_m)\bigr)\right|
    =
    \prod_{n\notin S}
    \left|\operatorname{Res}(\Phi_m,\Phi_n)\right|,
    $$
    so taking the product over $m\in S$ gives
    $
    [\mathcal A_S:F_{S^c}\mathcal A_S]
    =I_{\mathrm{cross}}(S).
    $
    Multiplicativity of index now gives the result.

\end{proof}

We also record the formula for resultants of cyclotomic polynomials \cite{Apostol}. For $m\ne n$, we have 
$$
\left|\operatorname{Res}(\Phi_m,\Phi_n)\right|
=
\begin{cases}
    \ell^{\varphi(\min\{m,n\})},
 &
 \displaystyle
 \frac{\max\{m,n\}}{\min\{m,n\}}
 \text{ is a positive power of the prime }\ell,
 \\[1ex]
 1,
 &\text{otherwise}.
\end{cases}
$$
It will be convenient to work with logarithmic indices. We therefore set 
$
I=\log I_{\mathrm{all}}
$
and
$$
I_{\mathrm{inc}}(S)
=
\log\!\left(
    I_{\mathrm{internal}}(S)I_{\mathrm{cross}}(S)
\right) = 
\sum_{\substack{\{m,n\}\subset\mathcal N\\
\{m,n\}\cap S\ne\varnothing}}
\log\left|\operatorname{Res}(\Phi_m,\Phi_n)\right|.
$$
Thus $I$ is the total logarithmic resultant weight of all pairs of
cyclotomic components, while $I_{\mathrm{inc}}(S)$ is the total weight
of those pairs having at least one endpoint in $S$.

\subsection{The glued module in dimension \texorpdfstring{$d$}{d}}
\label{subsec:exact-dimension}
From now on, assume that $1\in\mathcal N$. The order $\mathcal A$ has rank $C$ as a $\mathbb Z$-module, so
$\mathcal A^r$ has rank $Cr$. To work
in an arbitrary prescribed dimension $d$, we place the remaining coordinates in the $n=1$ component, on which the cyclic action
is trivial. This adjusts the dimension without introducing a new
cyclotomic component or changing the order of the cyclic action. Accordingly, write
$$
d=Cr+a,
\qquad
0\le a<C,
$$
where $r=\lfloor d/C\rfloor$, and define the glued module of dimension $d$
\begin{equation}
    L:=\mathcal A^r\oplus\mathbb Z^a.
    \label{eq:exact-module}
\end{equation}
It will serve as the arithmetic base module for the Haar probability ensemble. Define the multiplicities
$$
r_n
:=
\begin{cases}
    r+a,& n=1,\\
    r,& n\ne1.
\end{cases}
$$
Indeed, the rational cyclotomic decomposition
$
\mathcal A\otimes_{\mathbb Z}\mathbb Q
\simeq
\bigoplus_{n\in\mathcal N}K_n
$
gives
$$
L\otimes_{\mathbb Z}\mathbb Q
\simeq
K_1^{r+a}
\oplus
\bigoplus_{n\in\mathcal N\setminus\{1\}}K_n^r
=
\bigoplus_{n\in\mathcal N}K_n^{r_n}.
$$
Multiplication by $X$ defines an automorphism $\tau$ of
$\mathcal A$.  We let $\tau$ act diagonally on $\mathcal A^r$ and
trivially on the additional factor $\mathbb Z^a$.  Thus $L$ inherits the cyclic action from the cyclotomic order.

Passing to real vector spaces, set
$$
V_n
:=
\left(K_n\otimes_{\mathbb Q}\mathbb R\right)^{r_n},
\qquad
V:=\bigoplus_{n\in\mathcal N}V_n.
$$
Then
$
V\simeq L\otimes_{\mathbb Z}\mathbb R,
$
and under this identification $L$ is a full lattice in $V$. Concretely, realification turns the preceding algebraic construction into
an ordinary Euclidean lattice. In \cref{ex:two-component-gluing}, which
corresponds to $r=1$ and $a=0$, we have
$$
\mathcal A\otimes_{\mathbb Z}\mathbb R
\simeq
\bigl(\mathbb Q\oplus\mathbb Q(\zeta_3)\bigr)
\otimes_{\mathbb Q}\mathbb R
\simeq
\mathbb R\oplus\mathbb C
\simeq
\mathbb R^3,
$$
and $\mathcal A$ becomes a lattice of rank $3$ in this space. The stacked
planes in \cref{fig:two-component-gluing} represent the
$\mathbb R\oplus\mathbb C$ decomposition. The vertical direction is the $n=1$ component, while each horizontal plane is a copy of the $n=3$ component. 

For $S\subseteq\mathcal N$, define
$$
V_S:=\bigoplus_{n\in S}V_n,
\qquad
L_S:=L\cap V_S.
$$
Thus $L_S$ consists of vectors whose support is contained in $S$.
We write $L_S^\circ$ for the subset of vectors whose support is
exactly $S$. Put
$
\varepsilon_S:=\mathbf1_{\{1\in S\}}.
$
Since $\dim_{\mathbb R}(K_n\otimes_{\mathbb Q}\mathbb R)=c_n$, the real
dimension of $V_S$ is
\begin{equation}
    s_S
    =
    rc_S+a\varepsilon_S,
    \qquad
    c_S:=\sum_{n\in S}c_n.
    \label{eq:support-dimension}
\end{equation}

Choose on each $K_n\otimes_{\mathbb Q}\mathbb R$ a
$\tau$-invariant Euclidean structure, rescaled so that
$\mathcal O_n$ has covolume one, and equip $V$ with the corresponding
orthogonal direct sum structure. The index formulas of
\cref{subsec:resultants} give 
$$
\covol(\mathcal A)=e^I,
\qquad
\covol(\mathcal A\cap K_S)=e^{I_{\mathrm{inc}}(S)}.
$$
Taking $r$ copies, and noting that the additional
$\mathbb Z^a$-factor has covolume $1$, yields
\begin{equation*}
    \covol(L)=e^{rI}, \qquad 
    \covol(L_S)=e^{r I_{\mathrm{inc}}(S)}.
\end{equation*}

\subsection{Haar averaging for vectors with prescribed support}
\label{subsec:haar-averaging}

The gluing estimates control the covolumes of the support sublattices
$L_S$, but not directly the number of short vectors in a fixed lattice.
To pass from these covolumes to expected lattice-point counts, we average
the glued lattice $L$ over transformations of determinant one acting
separately on the cyclotomic components.

For sufficiently large $d$, we have $r_n\ge2$ for every
$n\in\mathcal N$.  Let
$$
G
:=
\prod_{n\in\mathcal N}
\operatorname{SL}_{r_n}
\bigl(K_n\otimes_{\mathbb Q}\mathbb R\bigr),
$$
acting componentwise on $V=\bigoplus_n V_n$, and let
$$
\Gamma
:=
\left\{
    \gamma\in\prod_{n\in\mathcal N}\operatorname{SL}_{r_n}(K_n):
    \gamma L=L
\right\}.
$$
The group $\Gamma$ is commensurable with
$
\prod_{n\in\mathcal N}\operatorname{SL}_{r_n}(\mathcal O_n),
$
and is therefore an arithmetic subgroup of
$
\mathbf G
=
\prod_{n\in\mathcal N}
\operatorname{Res}_{K_n/\mathbb Q}
\operatorname{SL}_{r_n}.
$ 
Since $\mathbf G$ is semisimple, it has no nontrivial
$\mathbb Q$-rational characters. Hence, by Borel--Harish-Chandra \cite[Thm.~12.3(c)]{BorelHarishChandra}, the quotient $X_L:=G/\Gamma$ has finite
invariant Haar measure.  We normalize the Haar measure on $X_L$ to a probability measure and
write $\mathbb E_g$ for integration over $g\in X_L$. 

Every $g\in G$ preserves the cyclotomic components and commutes with
$\tau$.  Thus $gL$ is again $\tau$-invariant with the same supports,
and, since $g$ has real determinant one,
$
\covol(gL)=\covol(L)=e^{rI}.
$
Thus $X_L$ is the probability space underlying our Haar ensemble of
lattices $gL$.

The corresponding Siegel mean-value formula for vectors of support $S$
is a version of the Siegel formula over number fields at finite level. See
\cite[Theorem~2.1]{AbuyaGargavaZhao} and
\cite[Theorem~1.2 and the subsequent discussion]{KimAdelicRogers}.

\begin{theorem}[Mean-value formula for vectors with prescribed support]
    \label{lem:mean}
    For every nonempty $S\subseteq\mathcal N$ and every nonnegative bounded
    compactly supported Borel function $f$ on $V_S$,
    \begin{equation}
        \mathbb E_g
        \sum_{v\in L_S^\circ} f(gv)
        =
        \frac{1}{\covol(L_S)}
        \int_{V_S} f(u)\,du.
        \label{eq:supportwise-siegel}
    \end{equation}
    Equivalently, for a nonnegative bounded compactly supported Borel
    function $F$ on $V$,
    $$
    \mathbb E_g
    \sum_{v\in L}F(gv)
    =
    F(0)
    +
    \sum_{\varnothing\ne S\subseteq\mathcal N}
    \frac{1}{\covol(L_S)}
    \int_{V_S}F(u)\,du.
    $$

\end{theorem}
\begin{proof}
    See Appendix \ref{app:mean-value-gluing}.
\end{proof}
Thus the arithmetic effect of the gluing may be computed precisely if
$\covol(L_S)$ can be computed. Large values of $I_{\mathrm{inc}}(S)$ therefore suppress the expected contribution from vectors of support $S$.

The resultant indices determine, for each support $S$, an arithmetic
penalty through the covolume of the corresponding support lattice.
These
penalties alone may not suppress every proper support strongly enough. We therefore
introduce an additional degree of freedom by rescaling the cyclotomic components relative to one another while preserving total volume. That is, we rescale the $n$-th cyclotomic component by $e^{y_n}$ and impose 
$
\sum_{n\in\mathcal N}c_ny_n=0.
$
On the first $Cr$ coordinates, the resulting componentwise rescaling has
determinant
$$
\exp\left(r\sum_{n\in\mathcal N} c_n y_n\right)=1.
$$ In the exact dimension $d=Cr+a$, the additional $a$ coordinates belong
to the $n=1$ component and contribute the factor $e^{ay_1}$. To remove this global volume factor, set
$
\eta=\frac{ay_1}{d}$ and $
\widetilde y_n=y_n-\eta.
$
Then
$$
\sum_{n\in\mathcal N}
\bigl(rc_n+a\mathbf1_{\{n=1\}}\bigr)\widetilde y_n
=
ay_1-d\eta
=
0.
$$
Hence the componentwise scaling
$
T_y
=
\bigoplus_{n\in\mathcal N}
e^{\widetilde y_n}\operatorname{Id}_{V_n}
$
has determinant one. Since both $g$ and $T_y$ have determinant one,
$
\covol(T_ygL)
=
\covol(L)
=
e^{rI}.
$
\subsection{Estimating the mean for a prescribed support}
\label{subsec:supportwise-mean}

For each positive integer $k$, let $v_k$ denote the volume of the
Euclidean unit ball in $\mathbb R^k$.
We therefore set
$
\lambda
=
\left(\frac{v_d}{e^{rI}}\right)^{1/d}
$
and define
\begin{equation}
    \Lambda_g^*
    =
    \lambda T_ygL.
    \label{eq:normalized-haar-lattice}
\end{equation}
Then $\covol(\Lambda_g^*)=v_d.$ Since both $g$ and $T_y$ preserve the cyclotomic components and commute with $\tau$, every $\Lambda_g^*$ is $\tau$-invariant and supports are unchanged.

Fix a nonempty $S\subseteq\mathcal N$, and write $s=s_S=rc_S+a\varepsilon_S.$ Let $B_s$ denote the Euclidean unit ball in the
$s$-dimensional space $V_S$ and define
$$
\mathcal L(S)
:=
\log\left(
    \mathbb E_g
    \left|\left\{
        x\in\Lambda_g^*\cap B_s:
        \operatorname{supp}(x)=S
    \right\}\right|
\right).
$$
Thus $\mathcal L(S)$ is the logarithm of the expected number of
vectors in the unit ball having exact support $S$.

\begin{lemma}[Logarithm of the mean for a prescribed support]
    \label{lem:supportwise-mean}
    For every nonempty $S\subseteq\mathcal N$,
    \begin{align}
        \mathcal L(S)
        ={}
        \log v_s
        -\frac{s}{d}\log v_d
        +\frac{s}{d}rI
        -rI_{\mathrm{inc}}(S)
        -
        \sum_{n\in S}
        \bigl(rc_n+a\mathbf1_{{n=1}}\bigr)\widetilde y_n.
        \label{eq:finite-support}
    \end{align}
\end{lemma}

\begin{proof}
    Since $\lambda T_y g$ preserves cyclotomic support, we have 
    $$
    \begin{aligned}
        \mathbb E_g
        \left|\left\{
            x\in\Lambda_g^*\cap B_s:
            \operatorname{supp}(x)=S
        \right\}\right| 
        =
        \mathbb E_g
        \sum_{v\in L_S^\circ}
        \mathbf1_{B_s}\bigl(\lambda T_ygv\bigr).
    \end{aligned}
    $$
    Applying \cref{lem:mean} with
    $
    f(u)
    =
    \mathbf1_{B_s}\bigl(\lambda T_yu\bigr)
    $
    gives
    $$
    \mathbb E_g
    \left|\left\{
        x\in\Lambda_g^*\cap B_s:
        \operatorname{supp}(x)=S
    \right\}\right|
    =
    \frac{
        \vol\bigl((\lambda T_y|_{V_S})^{-1}B_s\bigr)
    }{
        \covol(L_S)
    }.
    $$
    Now
    $$
    \det(T_y|_{V_S})
    =
    \exp\left(
        \sum_{n\in S}
        \bigl(rc_n+a\mathbf1_{\{n=1\}}\bigr)\widetilde y_n
    \right),
    $$
    while
    $
    \covol(L_S)=e^{rI_{\mathrm{inc}}(S)}.
    $
    Therefore
    $$
    \begin{aligned}
        \mathbb E_g
        \left|\left\{
            x\in\Lambda_g^*\cap B_s:
            \operatorname{supp}(x)=S
        \right\}\right|
        =
        v_s\lambda^{-s}
        \exp\left(
            -rI_{\mathrm{inc}}(S)
            -
            \sum_{n\in S}
            \bigl(rc_n+a\mathbf1_{\{n=1\}}\bigr)\widetilde y_n
        \right).
    \end{aligned}
    $$
    Taking logarithms and using
    $$
    -s\log\lambda
    =
    -\frac{s}{d}\log v_d
    +\frac{s}{d}rI
    $$
    gives \eqref{eq:finite-support}.
\end{proof}

The smaller dimension of $V_S$ does not by itself make the expected
number of vectors in the unit ball of exact support $S$ small. Indeed, the
normalization in the full dimension $d$ produces the geometric factor
$
v_{s} / v_d^{\,s/d},
$
whose logarithm is
$
\frac{s}{2}\log\frac{d}{s}+O(\log d).
$
This factor can be large, and is exponentially large when $s$ is a
fixed positive proportion of $d$.  This is the geometric contribution that appears in the exponent $\mathscr A(S)$ introduced below.

\subsection{The exponent \texorpdfstring{$\mathscr A(S)$}{A(S)}}

The formula for $\mathcal L(S)$ contains contributions of order
$r$, together with lower-order errors. We isolate these leading terms in the following quantity.
For every nonempty $S\subseteq\mathcal N$, define 
\begin{equation}
    \mathscr A(S)
    :=
    \frac{c_S}{2}\log\frac{C}{c_S}
    +\frac{c_S}{C}I
    -I_{\mathrm{inc}}(S)
    -\sum_{n\in S}c_ny_n.
    \label{eq:support-exponent}
\end{equation}
For the full support $S=\mathcal N$, the preceding formula gives
$
\mathscr A(\mathcal N)=0.
$
The four terms in $\mathscr A(S)$ have distinct origins. The first is
the contribution from the volume of a ball in a lower dimension, the second comes from
normalizing by the full covolume, the third is the gluing
penalty for the support $S$, and the last records the componentwise rescaling. Thus
$\mathscr A(S)$ measures the balance between the geometric contribution
of the support $S$ and the suppression coming from gluing and rescaling.

Under the hypotheses of the next lemma, $r \mathscr A(S)$ gives the
leading-order term in \eqref{eq:finite-support}.
Since our eventual goal is exponential suppression of contributions from
short vectors of proper support, we will in particular seek gluing data for which $\mathscr A(S)$ is
uniformly negative for every nonempty proper support
$S\subsetneq\mathcal N$. 

\begin{lemma} 
    \label{lem:support-exponent-comparison}
    Suppose that $I/C \ll\log C$ and $|y_1|\ll\log C.$ 
    Then there exists an absolute constant $K>0$ such that, for every nonempty
    $S\subseteq\mathcal N$,

    $$
    \mathcal L(S)
    \le
    r\mathscr A(S)+KC\log d.
    $$

\end{lemma}

\begin{proof}
    Starting from \eqref{eq:finite-support}, we compare each term with the
    corresponding term in $r\mathscr A(S)$. First, Stirling's formula gives, uniformly for $1\le s\le d$,
    $$
    \log v_s-\frac{s}{d}\log v_d
    \le
    \frac{s}{2}\log\frac ds+O(\log d).
    $$
    Since
    $$
    s=rc_S+a\varepsilon_S,
    \qquad
    d=Cr+a,
    \qquad
    0\le a<C,
    $$
    we obtain
    $$
    \frac{s}{2}\log\frac ds
    =
    r\frac{c_S}{2}\log\frac C{c_S}
    +O(C\log d).
    $$
    Next,
    $$
    \frac{s}{d}
    =
    \frac{rc_S+a\varepsilon_S}{Cr+a}
    =
    \frac{c_S}{C}
    +O\!\left(\frac1r\right).
    $$
    Therefore, using the assumed bound $I/C\ll\log C$, we get 
    $$
    \frac{s}{d}rI
    =
    r\frac{c_S}{C}I
    +O(I)
    =
    r\frac{c_S}{C}I
    +O(C\log d).
    $$
    The term involving $\covol(L_S)$ requires no approximation. For the componentwise scaling term, recall that
    $
    \widetilde y_n=y_n-\eta$ and 
    $
    \eta=\frac{ay_1}{d}.
    $
    Using
    $
    s
    =
    \sum_{n\in S}
    \bigl(rc_n+a\mathbf1_{\{n=1\}}\bigr)
    =
    rc_S+a\varepsilon_S,
    $
    we obtain
    \begin{align*}
        \sum_{n\in S}
        \bigl(rc_n+a\mathbf1_{{n=1}}\bigr)\widetilde y_n
        -
        r\sum_{n\in S}c_ny_n
        =
        a\varepsilon_Sy_1-\eta s
        =
        ay_1
        \left(
            \varepsilon_S-\frac{s}{d}
        \right).
    \end{align*}
    Since $0\le s/d\le1$ and $\varepsilon_S\in\{0,1\}$, we have 
    $
    \left|
    \varepsilon_S-\frac{s}{d}
    \right|
    \le1.
    $
    Thus
    $$
    \left|
    \sum_{n\in S}
    \bigl(rc_n+a\mathbf1_{\{n=1\}}\bigr)\widetilde y_n
    -
    r\sum_{n\in S}c_ny_n
    \right|
    \le
    a|y_1|.
    $$
    The assumption $|y_1|\ll\log C$ and $a<C$ implies that 
    $$
    \sum_{n\in S}
    \bigl(rc_n+a\mathbf1_{\{n=1\}}\bigr)\widetilde y_n
    =
    r\sum_{n\in S}c_ny_n
    +O(C\log d),
    $$
    uniformly in $S$. Substituting these estimates into \eqref{eq:finite-support} gives
    $$
    \mathcal L(S)
    \le
    r
    \left[
        \frac{c_S}{2}\log\frac C{c_S}
        +\frac{c_S}{C}I
        -I_{\mathrm{inc}}(S)
        -\sum_{n\in S}c_ny_n
    \right]
    +KC\log d
    $$
    for some absolute constant $K>0$, uniformly in $S$.

\end{proof}

\subsection{Admissible gluing data}
\label{subsec:gluing-criterion}

The previous lemma shows that the Haar mean for each support is governed, up to
lower-order terms, by the exponent $\mathscr A(S)$, provided
$$
\sum_{n\in\mathcal N} c_n y_n=0,
\qquad
\frac{I}{C}\ll\log C,
\qquad
|y_1|\ll\log C.
$$
To guarantee uniform suppression of every proper support, and to provide the quantitative bounds required later in the stochastic argument, we now introduce the following admissibility conditions.

\begin{definition}[Admissible family of gluing data]
    For each sufficiently large integer $d$, let
    $\mathcal N \subset\mathbb N$ be a finite index set with
    $1\in\mathcal N$ and $2\notin\mathcal N$, and put
    $$
    C=\sum_{n\in\mathcal N}c_n,
    \qquad
    M=\operatorname{lcm}\{n:n\in\mathcal N\}.
    $$
    Write
    $$
    r=\left\lfloor\frac{d}{C}\right\rfloor,
    \qquad
    d=C r+a,
    \qquad
    0\le a<C.
    $$
    For $n\in\mathcal N$, let $y_{n}$ be real numbers. We call the family $\left(\mathcal N,(y_{n})_{n\in\mathcal N}\right)$ admissible if there exists an absolute constant $\delta>0$ such that
    $$
    \sum_{n\in\mathcal N}c_n y_{n}=0,
    \qquad
    \mathscr A(S)\le-\delta
    $$
    for every nonempty proper $S\subsetneq\mathcal N$, together with 
    $$
    \frac{I}{C}\ll\log C,
    \qquad
    |y_{1}|\ll\log C,
    \qquad 
    r\gg C^3\log d,
    \qquad
    C\ge\log d,
    \qquad
    \log M\ll\log d.
    $$
    All implied constants and $\delta$ are required to be uniform in $d$. 
\end{definition}

The construction in \cref{sec:arithmetic-realizations} will produce an admissible
family satisfying $M/C \gg\log\log d$.

\subsection{Exponential suppression of proper supports}
\label{subsec:support-suppression}

Assume throughout this subsection that
$\bigl(\mathcal N,(y_n)_{n\in\mathcal N}\bigr)$ are admissible gluing
data for $d$, with $\mathscr A(S)\le-\delta$ for every nonempty proper
$S\subsetneq\mathcal N$ and a constant $\delta>0$. We show that admissibility makes the expected contribution of every
proper support exponentially small.

\begin{lemma}
    \label{lem:fixed-support}
    There exists $d_0$ such that, for every $d\ge d_0$ and every nonempty proper support $\varnothing\ne S\subsetneq\mathcal N, $ we have
    $$
    \mathcal L(S)
    \le
    -\frac34\delta r. 
    $$
    Equivalently,
    $$
    \mathbb E_g
    \left|\left\{
        x\in\Lambda_g^*\cap B_s:
        \operatorname{supp}(x)=S
    \right\}\right|
    \le
    e^{-3\delta r/4}.
    $$

\end{lemma}

\begin{proof}
    Since $S$ is a nonempty proper support, \cref{lem:support-exponent-comparison} and the uniform gluing condition give
    $$
    \mathcal L(S)
    \le
    -\delta r+KC\log d.
    $$
    Using the admissibility
    $
    r\gg C^3\log d
    $
    gives 
    $
    \frac{C\log d}{r}
    \ll
    \frac1{C^2}
    \to 0.
    $
    Thus, for all sufficiently large $d$, we have 
    $$
    KC\log d
    \le
    \frac{\delta}{4}r.
    $$
    This proves the claimed bound. 

\end{proof}

The preceding estimate is uniform in the support.  We now sum it over all
proper supports and allow the slight enlargement of the unit ball that will
be needed later in the stochastic argument.

\begin{corollary}
    \label{cor:proper-supports}
    Let $\rho>0$ satisfy
    \begin{equation}
        |\log\rho|
        =
        O\left(
            d^{-1}
            +
            \sqrt{\frac{\log d}{r}}
        \right).
        \label{eq:rho-range}
    \end{equation}
    Then, for some $c>0$, the expected number of nonzero vectors in
    $\Lambda_g^*\cap B_\rho$ that vanish in at least one cyclotomic component is at most
    $
    e^{-cr}.
    $
\end{corollary}

\begin{proof}
    For a support $S$ of real dimension $s$, replacing the unit ball by
    $B_\rho$ changes $\mathcal L(S)$ by
    $
    s\log\rho.
    $
    Since $s\le d$, \eqref{eq:rho-range} gives
    $
    |s\log\rho|
    \ll
    1+d\sqrt{\frac{\log d}{r}}.
    $
    Since
    $
    d=Cr+a
    \le C(r+1),
    $
    we obtain
    $
    |s\log\rho|
    \ll
    1+C\sqrt{r\log d}.
    $
    The admissibility condition 
    $
    r\gg C^3\log d
    $
    gives 
    $$
    \frac{
        C\sqrt{r\log d}
    }{r}
    =
    C\sqrt{\frac{\log d}{r}}
    \ll
    \frac1{\sqrt C}
    \to 0.
    $$
    Thus the change in $\mathcal L(S)$ is $o(r)$, uniformly in $S$.
    By \cref{lem:fixed-support}, for all sufficiently large $d$, the Haar
    expected number of vectors in $B_\rho$ with any fixed nonempty proper
    support is therefore, say, at most
    $
    e^{-\delta r/2}.
    $

    It remains only to count the possible supports.  Since
    $
    c_n=\varphi(n)\ge1
    $
    for every $n\in\mathcal N$, we have 
    $
    |\mathcal N|
    \le
    \sum_{n\in\mathcal N}c_n
    =
    C.
    $
    Hence the number of supports is at most
    $
    2^{|\mathcal N|}
    \le
    2^C.
    $
    Moreover,
    $$
    \frac Cr
    \ll
    \frac{1}{C^2\log d}
    \to 0
    $$
    gives 
    $
    C\log2=o(r).
    $
    Summing the estimates over all supports gives
    $
    2^C e^{-\delta r/2}
    \le
    e^{-cr}
    $
    for some absolute $c>0$ and all sufficiently large $d$.
\end{proof}

\section{Constructing admissible gluing data} 
\label{sec:arithmetic-realizations}

In this section we construct admissible gluing data. We begin by reformulating the support condition in terms of a weighted
resultant graph and a zero-mean potential.  After two simple warm-up
examples and a general product lemma, we construct divisor blocks whose
resultant graphs are Boolean cubes and choose the primes so as to obtain $M/C \gg\log\log d$ while satisfying all admissibility conditions. 
The subsequent argument requires only the existence statement furnished by this construction, which we summarize in the following theorem. 

\begin{theorem}
    \label{thm:arithmetic-realization}
    There exists an admissible family of gluing data
    $\left(\mathcal N,(y_{n})_{n\in\mathcal N}\right)$ such that  
    $$
    \frac{M}{C}\gg\log\log d.
    $$
\end{theorem}

\subsection{Resultant graphs and the support problem}

It is convenient to encode the cyclotomic resultants by a weighted graph. Define a weighted graph on vertex set $\mathcal N$ by assigning to distinct $m,n\in\mathcal N$ the weight
$$
w(m,n)
:=
\log\left|\Res(\Phi_m,\Phi_n)\right|.
$$
Thus
$$
I
=
\sum_{\{m,n\}\subseteq\mathcal N}w(m,n).
$$
For $n\in\mathcal N$, let
$$
\omega_n
:=
\sum_{\substack{m\in\mathcal N\\m\ne n}}w(m,n)
$$
be its weighted degree.  For $S\subseteq\mathcal N$, write
$W(S,S^c)$ for the total weight of the edges crossing from $S$ to
its complement. The incidence weight appearing in $\mathscr A(S)$ can then be
written in terms of vertex degrees and the boundary of $S$.
Indeed, an edge contained in $S$ contributes twice to
$\sum_{n\in S}\omega_n$, whereas a crossing edge contributes once.
Hence
\begin{equation}
    I_{\mathrm{inc}}(S)
    =
    \frac12\sum_{n\in S}\omega_n
    +
    \frac12W(S,S^c).
    \label{eq:incident-boundary}
\end{equation}
Next we reparametrize the volume-preserving component scales.  Given
$y=(y_n)_{n\in\mathcal N}$ satisfying
$
\sum_{n\in\mathcal N}c_ny_n=0,
$
define
\begin{equation*}
    v_n
    :=
    y_n-\frac IC+\frac{\omega_n}{2c_n}.
\end{equation*}
Since
$
\sum_{n\in\mathcal N}\omega_n=2I,
$
we have
\begin{equation}
    \sum_{n\in\mathcal N}c_nv_n=0.
    \label{eq:v-zero-mean}
\end{equation}
Conversely, every family $v=(v_n)$ satisfying
\eqref{eq:v-zero-mean} determines a volume-preserving scaling through
\begin{equation}
    y_n
    =
    \frac IC-\frac{\omega_n}{2c_n}+v_n.
    \label{eq:boundary-stretching}
\end{equation}
We regard $v$ as a vertex potential on the resultant graph. The condition
$ \sum_{n\in\mathcal N}c_nv_n=0 $
is the normalization corresponding to volume preservation.

In these variables $\mathscr A(S)$ takes the following 
form. Substituting \eqref{eq:incident-boundary} and \eqref{eq:boundary-stretching} into \eqref{eq:support-exponent}, the terms involving $I$ and $\sum_{n\in S}\omega_n$ cancel, giving 
\begin{equation}
    \mathscr A(S)
    =
    \frac{c_S}{2}\log\frac C{c_S}
    -
    \frac12W(S,S^c)
    -
    \sum_{n\in S}c_nv_n.
    \label{eq:boundary-potential-form}
\end{equation}

This form separates the three effects that determine the contribution of a
support. The first term is the geometric cost of restricting to the
lower-dimensional support $S$. The second is the penalty supplied by
the resultant edges crossing from $S$ to $S^c$. The final term is the
remaining freedom coming from the componentwise rescaling.

Thus the condition $\mathscr A(S)\le-\delta$ has been reduced to a combinatorial problem. We must choose
the index set $\mathcal N$, and hence its resultant graph,
together with a zero-mean potential $v$, so that $\mathscr A(S)\le -\delta$
for some absolute $\delta>0$ and every nonempty proper $S\subsetneq\mathcal N$, while
keeping $M/C$ as large as possible.

\subsection{Examples with a uniformly negative exponent}

For the following examples it is convenient to write
$$
\mathscr B(S)
:=
\frac{c_S}{2}\log\frac{C}{c_S}
-\frac12W(S,S^c),
\qquad
u_n:=c_nv_n,
\qquad
u(S):=\sum_{n\in S}u_n.
$$
Thus the remaining task is to choose a zero-sum potential that makes $\mathscr B(S)-u(S)$ uniformly negative.

\subsubsection{Example with two components}
Let $\mathcal N=\{1,p\},$ where $p\ge3$ is prime.  Then $C=p$, and the resultant graph is the
single weighted edge
$$
1\xleftrightarrow{\log p}p.
$$
There are only two proper supports, $S=\{1\}$ and $S^c=\{p\}.$  We have $\mathscr B(\{1\}) = 0$, whereas
$$
\mathscr B(\{p\})
=
\frac12
\left[
    (p-1)\log\frac{p}{p-1}-\log p
\right].
$$
Put
$$
t
:=
\frac12\bigl(
    \mathscr B(\{1\})+\mathscr B(\{p\})
\bigr)
=
\frac14
\left[
    (p-1)\log\frac{p}{p-1}-\log p
\right].
$$
Since
$$
(p-1)\log\frac{p}{p-1}<1,
$$
we have
$ t < -\delta_1 $
with
$
\delta_1:=\frac14(\log3-1)>0.
$
Now we choose $u_1=-t$ and $u_p=t.$ Then $u_1+u_p=0$ and
$$
\mathscr A(\{1\})
=
-u_1
=
t,
\qquad 
\mathscr A(\{p\})
=
\mathscr B(\{p\})-u_p
=
t.
$$
Hence
$
\mathscr A(S)\le-\delta_1
$
for every nonempty proper $S\subsetneq\mathcal N$.
Thus a single resultant edge already gives a uniformly negative bound for $\mathscr A(S)$. 
Its weight exceeds the geometric gain across the cut, while the
zero-mean potential merely redistributes this negative margin between
the two sides.  There is no efficiency gain in this example, however,
since $M=C=p.$

\subsubsection{Example with three components} 
\label{subsec:three-component-block}

We now add a third component.  Let $b,p$ be odd primes satisfying $6b\le p\le8b$ and take $\mathcal N=\{1,p,bp\}.$ The component degrees are $1, p-1$ and $ (b-1)(p-1),$
so
$
C=b(p-1)+1$ and $
M=bp.
$
The resultant graph is the weighted path
$$
1
\xleftrightarrow{\log p}
p
\xleftrightarrow{(p-1)\log b}
bp.
$$
Consider 
\begin{align*}
    t
    :=
    \frac12
    \left(
        \mathscr B(\{1\})
        +
        \mathscr B(\{p,bp\})
    \right)
    =
    \frac14
    \left[
        (C-1)\log\frac{C}{C-1}
        +\log\frac{C}{p^2}
    \right].
\end{align*}
Now
$$
(C-1)\log\frac{C}{C-1}<1,
\qquad
C\le bp.
$$
Hence
$$
t
<
\frac14\log\frac e6
=
-\delta_2,
\qquad
\delta_2:=\frac14\log\frac6e>0.
$$
We now choose the zero-mean potential.  Set
$$
\begin{aligned}
    u_1
    &:=
    \mathscr B(\{1\})-t,\\
    u_p
    &:=
    \mathscr B(\{1,p\})-\mathscr B(\{1\}),\\
    u_{bp}
    &:=
    -\mathscr B(\{1,p\})+t.
\end{aligned}
$$
Then
$
u_1+u_p+u_{bp}=0.
$
Moreover 
$$
\mathscr A(\{1\})
= 
\mathscr A(\{1,p\})
=
\mathscr A(\{p,bp\}) = t.
$$
The remaining three supports $\{p\},\{bp\}$ and $\{1,bp\}$ all cross the heavier edge
$
p\xleftrightarrow{(p-1)\log b}bp.
$
A direct estimate using $6b\le p\le8b$ gives
$
\mathscr A(S)\le t
$
for each of them. Thus the path with three vertices gives a uniformly negative bound for
$\mathscr A(S)$, as does the single edge. It also gives a gain in $M/C$,
$$
\frac MC
=
\frac{bp}{b(p-1)+1}
=
1+\frac{b-1}{b(p-1)+1}
>1.
$$

\subsection{Preserving the negative bound under products} 

We record facts about products that will be used below. The first lemma is a direct consequence of the standard tensorization inequality for entropy. The second shows
that, when the boundary and potential terms also have product structure,
uniform negative bounds for nonempty proper subsets are preserved under products.

For $1\le j\le J$, let
$X_j$ be a finite set with positive vertex weights $c_j$. 
Put 
$$
X:=\prod_{j=1}^JX_j,
\qquad
c(x):=\prod_{j=1}^Jc_j(x_j),
\qquad 
C_j:=\sum_{x_j\in X_j}c_j(x_j).
\qquad
C:=\prod_{j=1}^JC_j.
$$
For $S\subseteq X$, write
$
c(S):=\sum_{x\in S}c(x).
$
For each $j$, let
$
X_{-j}:=\prod_{i\ne j}X_i.
$
If 
$
u=(u_i)_{i\ne j}\in X_{-j},$ define  
the $j$-th fiber of $S$ over $u$ by $$ S_{j,u} := \left\{ x_j\in X_j: (u_1,\ldots,u_{j-1},x_j,u_{j+1},\ldots,u_J)\in S \right\}. $$
We also put $$ c_{-j}(u):=\prod_{i\ne j}c_i(u_i), \qquad c_j(S_{j,u}) := \sum_{x_j\in S_{j,u}}c_j(x_j). $$ 

For a probability measure $\mu$ and a nonnegative function $f$, define its entropy by 
$$
\Ent_\mu(f)
:=
\mathbb E_\mu[f\log f]
-
\mathbb E_\mu[f]\log\mathbb E_\mu[f].
$$
For an indicator $f=\mathbf1_A$, this becomes
$
\Ent_\mu(\mathbf1_A)
=
\mu(A)\log\frac1{\mu(A)}.
$
The following estimate is an immediate consequence of the standard
tensorization inequality for entropy. 

\begin{lemma} 
    \label{lem:entropy-tensorization}
    For every $S\subseteq X$ we have 
    $$
    \frac{c(S)}2\log\frac{C}{c(S)}
    \le
    \sum_{j=1}^J
    \sum_{u\in X_{-j}}
    c_{-j}(u)\,
    \frac{c_j(S_{j,u})}{2}
    \log\frac{C_j}{c_j(S_{j,u})},
    $$
    where the summand is interpreted as $0$ when
    $S_{j,u}=\varnothing$.
\end{lemma}

\begin{proof}
    Let
    $
    \mu_j(x_j)=\frac{c_j(x_j)}{C_j}
    $
    be the probability measure on $X_j$ and let
    $
    \mu=\bigotimes_{j=1}^J\mu_j
    $
    be the product measure on $X$. By the product structure we have 
    $$
    \mu(x)=\prod_{j=1}^J\frac{c_j(x_j)}{C_j} = \frac{c(x)}{C},
    \qquad
    \mu(S)=\frac{c(S)}{C}.
    $$
    The entropy tensorization inequality
    \cite[Eq.~(1.2)]{CaputoMenzTetali2015}, applied to $f=1_S$, is
    $$
    \Ent_\mu(1_S)
    \le
    \sum_{j=1}^J
    \mu\!\left[\Ent_{\mu_j}(1_S)\right].
    $$
    Since
    $
    \Ent_\mu(1_S)
    =
    \mu(S)\log(1/\mu(S)),
    $
    the left-hand side becomes
    $$
    \frac C2\,\Ent_\mu(1_S)
    =
    \frac{c(S)}{2}\log\frac{C}{c(S)}.
    $$
    For $u\in X_{-j}$ we have 
    $$
    \Ent_{\mu_j}(1_{S_{j,u}})
    =
    \frac{c_j(S_{j,u})}{C_j}
    \log\frac{C_j}{c_j(S_{j,u})},
    \qquad
    \mu_{-j}(u)
    =
    \frac{c_{-j}(u)}{C/C_j}.
    $$
    Hence
    $$
    \begin{aligned}
        \mu\!\left[\Ent_{\mu_j}(1_S)\right]
        =
        \sum_{u\in X_{-j}}
        \mu_{-j}(u)\,
        \Ent_{\mu_j}(1_{S_{j,u}}) 
        =
        \frac1C
        \sum_{u\in X_{-j}}
        c_{-j}(u)c_j(S_{j,u})
        \log\frac{C_j}{c_j(S_{j,u})}.
    \end{aligned}
    $$
    Multiplying the tensorization inequality by $C/2$ gives the claim.
\end{proof}

We now incorporate the graph structure.  For a finite weighted graph
$Y$, with vertex weights $c(y)$, total weight
$
C_Y:=\sum_{y\in Y}c(y)
$
and potential $v$. For $U\subseteq Y$, write
$c(U):=\sum_{y\in U}c(y)$, and let $W_Y(U,U^c)$ denote the total
weight of edges joining $U$ to its complement. Define
$$
\mathcal E_Y(U)
:=
\frac{c(U)}2\log\frac{C_Y}{c(U)}
-\frac12W_Y(U,U^c)
-\sum_{y\in U}c(y)v(y).
$$
We now prove that the bound $\mathcal E_{X_j}(U)\le-\delta$ implies
$\mathcal E_X(S)\le-\delta$ for every nonempty proper $S\subsetneq X$. 

\begin{lemma} 
    \label{lem:product-support}
    For $1\le j\le J$, let $X_j$ be finite weighted graphs with positive
    integer vertex weights $c_j$, edge weight functions $w_j$, total
    weights $C_j$, and potentials $v_j$ satisfying
    $
    \sum_{x_j\in X_j}c_j(x_j)v_j(x_j)=0.
    $
    Suppose that, for some $\delta>0$,
    $
    \mathcal E_{X_j}(U)\le-\delta
    $
    for every $j$ and every nonempty proper $U\subsetneq X_j$.

    Equip
    $
    X=\prod_{j=1}^JX_j
    $
    with the vertex weights
    $
    c(x)=\prod_{j=1}^Jc_j(x_j)
    $
    and potential
    $
    v(x)=\sum_{j=1}^Jv_j(x_j).
    $
    Give $X$ the product weighted graph in which two vertices are adjacent if they
    differ in exactly one coordinate, say the $j$-th, and such an edge has
    weight
    $$
    \left(\prod_{i\ne j}c_i(x_i)\right)
    w_j(x_j,x_j').
    $$
    Then $v$ has weighted mean zero and
    $
    \mathcal E_X(S)\le-\delta
    $
    for every nonempty proper $S\subsetneq X$.
\end{lemma}

\begin{proof}
    We have 
    $$
    \begin{aligned}
        \sum_{x\in X}c(x)v(x)
&=
\sum_{j=1}^J
\left(\prod_{i\ne j}C_i\right)
\sum_{x_j\in X_j}c_j(x_j)v_j(x_j)
=0,
    \end{aligned}
    $$
    so the product potential has weighted mean zero.

    Now fix $S\subseteq X$. \cref{lem:entropy-tensorization} gives 
    $$
    \frac{c(S)}2\log\frac{C}{c(S)}
    \le
    \sum_{j=1}^J
    \sum_{u\in X_{-j}}
    c_{-j}(u)
    \frac{c_j(S_{j,u})}{2}
    \log\frac{C_j}{c_j(S_{j,u})}.
    $$
    By the definition of the product graph, the boundary term decomposes as
    $$
    W_X(S,S^c)
    =
    \sum_{j=1}^J
    \sum_{u\in X_{-j}}
    c_{-j}(u)
    W_j(S_{j,u},S_{j,u}^c).
    $$
    Likewise we have 
    $$
    \sum_{x\in S}c(x)v(x)
    =
    \sum_{j=1}^J
    \sum_{u\in X_{-j}}
    c_{-j}(u)
    \sum_{x_j\in S_{j,u}}
    c_j(x_j)v_j(x_j).
    $$
    Combining these three relations gives
    $$
    \mathcal E_X(S)
    \le
    \sum_{j=1}^J
    \sum_{u\in X_{-j}}
    c_{-j}(u)\,
    \mathcal E_{X_j}(S_{j,u}).
    $$
    If $S_{j,u}=\varnothing$, then
    $\mathcal E_{X_j}(S_{j,u})=0$.  The same is true when
    $S_{j,u}=X_j$, by the zero-mean condition on $v_j$.
    Every nonempty proper fiber, on the other hand, satisfies
    $
    \mathcal E_{X_j}(S_{j,u})\le-\delta.
    $
    Suppose now that $S$ is nonempty and proper.  Choose
    $x\in S$ and $x'\notin S$, and replace the coordinates of $x$ by
    those of $x'$ one at a time.  At some step one passes from $S$ to
    $S^c$.  The corresponding fiber $S_{j,u}$ is therefore nonempty and
    proper.  Since the vertex weights are positive integers,
    $
    c_{-j}(u)\ge1.
    $
    Hence at least one term in the preceding sum is at most $-\delta$, while
    all the others are nonpositive.  Therefore
    $
    \mathcal E_X(S)\le-\delta.
    $
\end{proof}

As a first application of \cref{lem:product-support}, we may take products
of the blocks with three components above. The bound
$\mathscr A(S)\le-\delta$ remains uniform, while

$$
\frac MC
=
\prod_j
\frac{b_jp_j}{b_j(p_j-1)+1}.
$$
By choosing the primes $b_j,p_j$ in suitable disjoint ranges, these
individual gains accumulate and give admissible gluing data satisfying 
$
M/C \gg(\log\log d)^\kappa
$
for some absolute constant $0<\kappa<1$.  We will not pursue this 
bound further, since the construction using blocks of divisors below yields the full logarithmic gain
$M/C\gg\log\log d$.

\subsection{Divisor blocks and the full logarithmic gain}
\label{subsec:optimal-arithmetic}

Products of the blocks with three components preserve the bound
$\mathscr A(S)\le-\delta$, but in this
construction each prime $p_j$ is paired with its own auxiliary prime $b_j$, while 
$$
\frac{b_jp_j}{b_j(p_j-1)+1}
\asymp
\frac{p_j}{p_j-1}.
$$
Thus $b_j$ enlarges both $M$ and $C$ by essentially the same factor,
and therefore contributes essentially nothing to the gain in $M/C$. The construction below avoids this loss by grouping many
primes into a single divisor block, while keeping the total degree $C$ comparable to $\varphi(M)$.

\subsubsection{A cube of divisors with one additional vertex}

Let $\mathcal P$ be a finite set of primes and put
$
N=\prod_{p\in\mathcal P}p.
$
Let $q\notin \mathcal P$ be another prime.  Define the divisor block 
\begin{equation}
    \mathcal D(q;N)
    :=
    \{n:n\mid N\}\cup \{qN \}.
    \label{eq:divisor-block}
\end{equation}
For $\mathcal P=\{p\}$, this is exactly the block with three components
$\{1,p,qp\}$. Since $N$ is squarefree, every divisor of $N$ is uniquely of the
form
$
n_A=\prod_{p\in A}p$ for 
$
A\subset \mathcal P.
$
Thus the divisors of $N$ are naturally identified with the vertices
of the Boolean cube $2^{\mathcal P}$. Under this identification,
multiplying or dividing by one prime corresponds to moving along one
edge of the cube. In the resultant graph, the
additional vertex $qN$ is adjacent only to the vertex $N$.  

Put
$$
P
:=
\frac N{\varphi(N)}
=
\prod_{p\in \mathcal P}\frac p{p-1},
\qquad
a:=\frac1{q-1}.
$$
The total degree of the block $\mathcal D(q;N)$ is
\begin{align}
    C_{\mathrm{loc}}
    :=
    \sum_{m\in\mathcal D(q;N)}\varphi(m)
    =
    \sum_{n\mid N}\varphi(n)+\varphi(qN) 
    =
    N+(q-1)\varphi(N) 
    =
    \varphi(qN)(1+aP).
    \label{eq:local-degree}
\end{align}
We now describe the weighted graph. We normalize all vertex and edge weights by the degree
$\varphi(qN)$ of the top component.  Define

$$
\rho(B)
:=
\prod_{p\in B}\frac1{p-1},
\qquad
\rho(\varnothing)=1.
$$
Then the normalized degree of $n_A$ is
\begin{equation}
    \lambda(A)
    :=
    \frac{\varphi(n_A)}{\varphi(qN)}
    =
    a \rho(\mathcal P \setminus A),
    \label{eq:local-vertex-weight}
\end{equation}
while
$
\lambda(qN)=1.
$
Apostol's resultant formula shows that there are exactly two types of
edges in the block.  If $p\in \mathcal P  \setminus A $, then
$
n_A\longleftrightarrow pn_A
$
is an edge of normalized weight $\lambda(A)\log p.$ In addition,
$
N\longleftrightarrow qN
$
is an edge of normalized weight $a\log q.$ An example of the local resultant graph, with
$N=p_1p_2p_3$, is illustrated in \cref{fig:divisor-block}.

\begin{figure}[!ht]
    \centering
    \begin{tikzpicture}[
        scale=1,
        every node/.style={font=\small},
        vertex/.style={draw, rounded corners=2pt, fill=white, inner sep=3pt},
        topvertex/.style={draw, rounded corners=2pt, fill=gray!10, inner sep=3pt}
        ]

        \node[vertex] (1) at (0,0) {$1$};

        \node[vertex] (p1) at (-3,1.6) {$p_1$};
        \node[vertex] (p2) at (0,1.6) {$p_2$};
        \node[vertex] (p3) at (3,1.6) {$p_3$};

        \node[vertex] (p1p2) at (-3,3.2) {$p_1p_2$};
        \node[vertex] (p1p3) at (0,3.2) {$p_1p_3$};
        \node[vertex] (p2p3) at (3,3.2) {$p_2p_3$};

        \node[vertex] (N) at (0,4.8) {$N=p_1p_2p_3$};

        \node[topvertex] (F) at (0,6.5) {$F=qN$};

        \draw (1)--(p1) (1)--(p2) (1)--(p3);

        \draw (p1)--(p1p2) (p1)--(p1p3);
        \draw (p2)--(p1p2) (p2)--(p2p3);
        \draw (p3)--(p1p3) (p3)--(p2p3);

        \draw (p1p2)--(N) (p1p3)--(N) (p2p3)--(N);

        \draw[thick] (N)--(F);

    \end{tikzpicture}

    \caption{A divisor block $\mathcal D(q;N)$ for
        $N=p_1p_2p_3$.  The divisors of $N$ form a Boolean cube.
        A cube edge obtained by adjoining $p$ to $n_A$ has normalized
        weight $\lambda(A)\log p$.  The additional vertex $F=qN$ is
    connected only to $N$, by an edge of normalized weight $a\log q$.}
    \label{fig:divisor-block}
\end{figure}
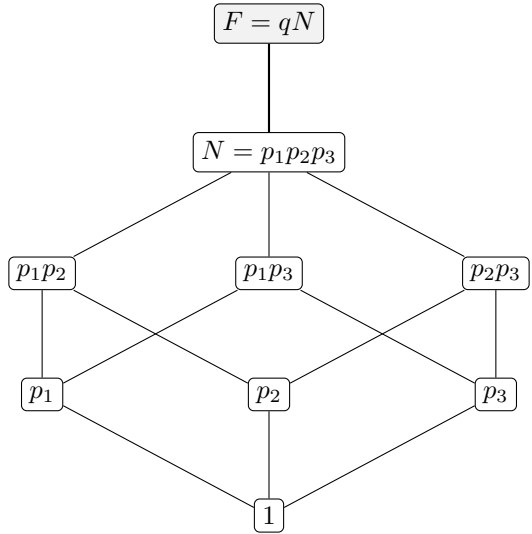

\subsubsection{Controlling all cuts of the divisor cube}

The boundary formulation \eqref{eq:boundary-potential-form} shows that
we need to compare the geometric cost of a support with the resultant
weight crossing its boundary.  For the divisor block, the lower part of
the resultant graph is a Boolean cube, so the next step is to obtain
such a comparison uniformly for every subset of this cube.

It is convenient to relabel a divisor by the
primes it omits.  Thus, if $A\subseteq\mathcal P$ records the primes
present in $n_A$, we write
$
B=\mathcal P\setminus A,
$
and the normalized vertex weight of $n_A$ is then $
a\,\rho(B).$ For $\mathcal U\subseteq2^{\mathcal P}$, let
$
u(\mathcal U)
:=
\sum_{B\in\mathcal U}\rho(B)
$
be its total vertex weight. For $p\notin B$, assign the cube edge
$
B\longleftrightarrow B\cup\{p\}
$
the weight
$
\rho(B\cup\{p\})\log p,
$
and let $E_{\mathcal P}(\mathcal U)$ be the total weight of the edges
crossing from $\mathcal U$ to its complement. The next lemma shows that every subset of the divisor cube has enough boundary weight to control the term $u\log(P/u)$ from the support dimension.

\begin{lemma} 
    \label{lem:divisor-cube-entropy}
    For every $\mathcal U\subseteq2^{\mathcal P}$ we have 
    \begin{equation*}
        E_{\mathcal P}(\mathcal U)
        \ge
        u(\mathcal U)\log\frac{P}{u(\mathcal U)}.
    \end{equation*}
\end{lemma}

\begin{proof}
    Set
    $
    \mu(B):=\frac{\rho(B)}{P}.
    $
    Then
    $$
    \mu(B)
    =
    \prod_{p\in B}\frac1p
    \prod_{p\in\mathcal P\setminus B}\frac{p-1}{p},
    $$
    so $\mu$ is obtained by including each $p\in\mathcal P$
    independently with probability $1/p$.
    Let $f=\mathbf1_{\mathcal U}$.  Entropy tensorization gives
    $$
    \Ent_\mu(f)
    \le
    \sum_{p\in\mathcal P}
    \mathbb E_{\mu_{-p}}\!\left[\Ent_{\mu_p}(f)\right].
    $$
    For fixed $p$, a choice of the remaining primes is a set
    $B\subseteq\mathcal P\setminus\{p\}$.  Varying whether $p$ is
    included gives the cube edge
    $
    B\longleftrightarrow B\cup\{p\}.
    $
    If this edge does not cross the boundary of $\mathcal U$, then $f$
    has the same value at its two endpoints and hence contributes zero
    entropy.  If it does cross the boundary, then $f$ takes the values
    $0$ and $1$. The endpoint at which
    $f=1$ has probability either $1/p$ or $(p-1)/p$
    and in either case its entropy is at most
    $
    (\log p)/p.
    $ 

    The probability of the choice $B$ outside the $p$-coordinate is
    $$
    \mu_{-p}(B)
    =
    \prod_{p'\in B}\frac1{p'}
    \prod_{p'\in\mathcal P\setminus(B\cup\{p\})}\frac{p'-1}{p'}.
    $$
    Hence
    $$
    \Ent_\mu(f)
    \le
    \sum_{\substack{p\in\mathcal P,\,
            B\subseteq\mathcal P\setminus\{p\}\\
            B\leftrightarrow B\cup\{p\}
    \text{ crosses }\partial\mathcal U}}
    \mu_{-p}(B)\frac{\log p}{p}.
    $$
    But
    $$
    P\,\mu_{-p}(B)\frac{\log p}{p}
    =
    \rho(B\cup\{p\})\log p,
    $$
    which is exactly the weight of the edge
    $B\leftrightarrow B\cup\{p\}$. Summing over all such crossing edges gives
    $
    P\,\Ent_\mu(f)
    \le
    E_{\mathcal P}(\mathcal U).
    $
    Combining this with 
    $   P\Ent_\mu(\mathbf1_{\mathcal U})
    =
    u(\mathcal U)\log\frac{P}{u(\mathcal U)}
    $
    proves the claim.
\end{proof}

We will also need the following boundary estimate.

\begin{lemma}
    \label{lem:cube-linear-bound}
    If $\varnothing\notin\mathcal U\subseteq2^{\mathcal P}$, then
    $$
    E_{\mathcal P}(\mathcal U)
    \ge
    \frac{\ell}{P}\,u(\mathcal U),
    \qquad
    \ell:=\min_{p\in\mathcal P}\log p.
    $$
\end{lemma}

\begin{proof}
    Fix an ordering of the primes in $\mathcal P$.  For each
    $B\in\mathcal U$, consider the path from $B$ to $\varnothing$
    obtained by removing the primes of $B$ one at a time in this order. Give this path weight $\rho(B)$.  Since
    $\varnothing\notin\mathcal U$, every such path crosses the boundary
    of $\mathcal U$.

    Consider a cube edge obtained by removing one prime,
    $$
    D\longleftrightarrow D\setminus\{p\},
    \qquad D\subseteq\mathcal P,\quad p\in D.
    $$  Only paths starting at sets $B\supseteq D$ can use this edge. Since every $B\supseteq D$ is uniquely of the form
    $B=D\cup C$, with $C\subseteq\mathcal P\setminus D$, we have
    $$
    \sum_{B\supseteq D}\rho(B)
    =
    \rho(D)
    \prod_{p'\in\mathcal P\setminus D}\frac {p'}{p'-1}
    \le
    P\rho(D).
    $$ 
    On the other hand, since the edge has weight
    $
    \rho(D)\log p\ge \ell\rho(D),
    $
    we obtain
    $$
    P\rho(D)
    \le
    \frac{P}{\ell}
    \bigl(\text{weight of the edge}\bigr).
    $$  Since every path crosses the boundary of
    $\mathcal U$, summing over the boundary edges gives
    $$
    u(\mathcal U)
    \le
    \frac{P}{\ell}E_{\mathcal P}(\mathcal U),
    $$
    as claimed. 
\end{proof}

\subsubsection{The local potential}

We now pass from the divisor cube to the full block
$\mathcal D(q;N)$, which also contains the additional vertex $F:=qN$.
The only new resultant edge is
$
N\longleftrightarrow qN.
$
We now use the remaining freedom in the zero-mean potential $v$ from
\eqref{eq:boundary-potential-form} to extend the bound from the divisor
cube to the full block.

For the rest of this subsection, fix
$
\delta_0:=\frac1{100} 
$
and assume that
\begin{equation}
    P\le\frac54,
    \qquad
    \ell\ge8(\log q+1).
    \label{eq:good-local-block}
\end{equation} 
The primes will be chosen later so that these conditions hold. 

The total
normalized weight of the full block is $
Z=1+aP.
$
Put
\begin{equation}
    \Theta
    :=
    \log\frac{Z}{aP}
    =
    \log\left(1+\frac{q-1}{P}\right).
    \label{eq:Theta}
\end{equation}
In particular we have 
$
0<\Theta\le\log q.
$
For $A\subseteq \mathcal P $, define
\begin{equation}
    \begin{aligned}
        v_A :=
        \frac{\Theta}{2}+\delta_0
        -\frac{\log q}{2}\mathbf1_{\{A=\mathcal P\}}, \qquad
        v_F :=
        -\frac a2(P\Theta-\log q)-\delta_0aP.
    \end{aligned}
    \label{eq:local-potential}
\end{equation}
The second line is chosen precisely so that the weighted mean vanishes,
\begin{equation}
    \sum_{A\subseteq \mathcal P}\lambda(A)v_A+v_F=0.
    \label{eq:local-potential-zero}
\end{equation}

The following is the main local estimate.  It shows that the potential
\eqref{eq:local-potential} gives the uniform negative bound in \eqref{eq:one-block-support}
on the entire block $\mathcal D(q;N)$.  The proof is arranged so that supports not containing $F$ reduce
directly to the cube boundary estimate.  For supports containing $F$,
we pass to the complement.  If the edge $N$--$F$ crosses the
boundary, its weight gives the required negative bound. Otherwise we use the
stronger linear boundary estimate on the divisor cube.

\begin{lemma} 
    \label{lem:one-divisor-block}
    Assume \eqref{eq:good-local-block} and that $q$ is sufficiently large.
    Then for every $U\subseteq\mathcal D(q;N)$ we have 
    \begin{equation}
        \frac{c_U}{2}\log\frac{C_{\mathrm{loc}}}{c_U}
        -\frac12W(U,U^c)
        -\sum_{x\in U}\varphi(x)v_x
        \le
        -\delta_0\min\{c_U,C_{\mathrm{loc}}-c_U\},
        \label{eq:one-block-support}
    \end{equation}
    where $
    c_U:=\sum_{x\in U}\varphi(x).
    $
    The first term is interpreted as zero when $c_U=0$.
\end{lemma}

\begin{proof}
    We normalize by $\varphi(qN)$ and write
    $$
    s(U):=\frac{c_U}{\varphi(qN)}
    =\sum_{x\in U}\lambda(x),
    \qquad
    Z:=\frac{C_{\mathrm{loc}}}{\varphi(qN)},
    \qquad
    E(U):=\frac{W(U,U^c)}{\varphi(qN)}.
    $$
    Thus it is enough to prove
    \begin{equation}
        \frac{s(U)}2\log\frac{Z}{s(U)}
        -\frac12E(U)
        -\sum_{x\in U}\lambda(x)v_x
        \le
        -\delta_0\min\{s(U),Z-s(U)\}.
        \label{eq:normalized-one-block}
    \end{equation}
    The cases $U=\varnothing$ and
    $U=\mathcal D(q;N)$ are immediate. Suppose first that $F\notin U$, so $U$ consists only of divisors
    of $N$.  For $n\mid N$, let
    $
    B(n):=\{p\in\mathcal P:p\nmid n\}
    $
    be the set of primes omitted from $n$, and put
    $
    \mathcal U:=\{B(n):n\in U\}$ and  $u:=u(\mathcal U).$ 
    Since $\lambda(n)=a\rho(B(n))$, we have $
    s(U)=au.
    $
    Put 
    $
    \varepsilon:=\mathbf1_{\{N\in U\}}.
    $
    The boundary of $U$ consists of edges of the divisor cube with one endpoint in $U$ and one outside
    $U$, and possibly the additional edge $N$--$F$. The normalized
    weights of the cube edges are $a$ times the weights used in
    $E_{\mathcal P}$, while the normalized weight of $N$--$F$ is
    $a\log q$.  Hence
    $$
    E(U)
    =
    aE_{\mathcal P}(\mathcal U)
    +a\varepsilon\log q.
    $$ 
    Moreover, by \eqref{eq:local-potential}, we have 
    $$
    \sum_{x\in U}\lambda(x)v_x
    =
    \frac{s(U)\Theta}{2}
    +\delta_0s(U)
    -\frac{a\varepsilon\log q}{2}.
    $$
    We also note that 
    $$
    \log\frac{Z}{s(U)}
    =
    \Theta+\log\frac{P}{u}.
    $$
    Substituting these identities into the left-hand side of
    \eqref{eq:normalized-one-block}, the terms involving $\Theta$ and
    $\varepsilon\log q$ cancel, and we obtain
    \begin{align} \label{eq:non-pos}
        \frac a2
        \left(
            u\log\frac Pu-E_{\mathcal P}(\mathcal U)
        \right)
        -\delta_0s(U).
    \end{align}
    By Lemma~\ref{lem:divisor-cube-entropy}, the term in parentheses is nonpositive.
    Hence  \eqref{eq:non-pos} is at most
    $
    -\delta_0s(U).
    $
    Since $s(U)\le aP<1$ while $Z-s(U)\ge1$, we have
    $$
    \min\{s(U),Z-s(U)\}=s(U),
    $$
    which proves \eqref{eq:normalized-one-block} when $F\notin U$.

    Now suppose that $F\in U$ and let $
    V:=\mathcal D(q;N)\setminus U.
    $
    Then $V\neq\varnothing$ and $F\notin V$. Put
    $$
    s:=s(V),
    \qquad
    \mathcal V:=\{B(n):n\in V\},
    \qquad
    \varepsilon:=\mathbf1_{\{N\in V\}},
    $$
    where $B(n)$ is as above.  Since the same edges cross the
    boundaries of $U$ and $V$, we have 
    $
    E(U)=E(V).
    $ Moreover,
    \eqref{eq:local-potential-zero} gives
    $$
    \sum_{x\in U}\lambda(x)v_x
    =
    -\sum_{x\in V}\lambda(x)v_x.
    $$
    Using
    $$
    (Z-s)\log\frac{Z}{Z-s}\le s,
    $$
    and applying the formulas from the first case to $V$, we obtain 
    \begin{equation}
        \begin{aligned}
&\frac{Z-s}{2}\log\frac{Z}{Z-s}
-\frac12E(V)
+\sum_{x\in V}\lambda(x)v_x \\
&\qquad\le
\frac s2
-\frac a2E_{\mathcal P}(\mathcal V)
+\frac{s\Theta}{2}
+\delta_0s
-a\varepsilon\log q.
        \end{aligned}
        \label{eq:complement-local-bound}
    \end{equation}
    If $\varepsilon=0$, then $N\notin V$, and hence
    $\varnothing\notin\mathcal V$. 
    \cref{lem:cube-linear-bound} then gives 
    $$
    E_{\mathcal P}(\mathcal V)
    \ge
    \frac{\ell}{P}u(\mathcal V)
    =
    \frac{\ell}{aP}s.
    $$
    Therefore the right-hand side of
    \eqref{eq:complement-local-bound} is at most
    $$
    \frac s2
    \left(
        1+\Theta+2\delta_0-\frac{\ell}{P}
    \right)
    \le
    -\delta_0s,
    $$
    where we used \eqref{eq:good-local-block} and
    $\Theta\le\log q$.

    Finally, suppose that $\varepsilon=1$.  Then the edge $N$--$F$
    crosses the boundary of $V$.  Since $E_{\mathcal P}(\mathcal V)\ge0$, we may discard the
    nonpositive second term in \eqref{eq:complement-local-bound}. Using $s\le aP$, we get
    $$
    \eqref{eq:complement-local-bound}
    \le
    \frac{aP}{2}(1+\Theta+2\delta_0)-a\log q.
    $$
    Since $P\le5/4$ and $\Theta\le\log q$, this is at most
    $
    -\delta_0aP
    \le
    -\delta_0s
    $
    for all sufficiently large $q$. This proves \eqref{eq:normalized-one-block}, and hence
    \eqref{eq:one-block-support}.
\end{proof}

\subsubsection{Products of blocks}

We now take products of the divisor blocks. Let
$
(q_j,N_j)
$
with $1\le j\le J$ be blocks whose prime factors are pairwise disjoint. Put
$$
\mathcal D_j
=
\mathcal D(q_j;N_j)
=
\{n:n\mid N_j\}\cup\{q_jN_j\}.
$$
Define
\begin{equation*}
    M
    :=
    \prod_{j\le J}q_jN_j,
    \qquad
    \mathcal N
    :=
    \left\{
        \prod_{j\le J}n_j:
        n_j\in\mathcal D_j
    \right\}.
\end{equation*}
Because the blocks are coprime we have 
$
\varphi\left(\prod_jn_j\right)
=
\prod_j\varphi(n_j).
$
In particular, writing
$
C_j=\sum_{n_j\in\mathcal D_j}\varphi(n_j),
$
we have 
$
C=\prod_{j\le J}C_j.
$
For each $j$, let $v_j$ be the potential from
\eqref{eq:local-potential}.  By \cref{lem:one-divisor-block}, every
nonempty proper $U\subsetneq\mathcal D_j$ satisfies
$$
\mathcal E_{\mathcal D_j}(U)
\le
-\delta_0
\min\{c_j(U),C_j-c_j(U)\}
\le
-\delta_0,
$$
since the vertex weights are positive integers.  Moreover
$
\sum_{x_j\in\mathcal D_j}\varphi(x_j)v_j(x_j)=0.
$
We verify that the resultant graph has the product
structure required by \cref{lem:product-support}.
Write $n=\prod_j n_j$ and $m=\prod_j m_j$,
where $n_j,m_j\in\mathcal D_j$.
Since the prime sets of distinct blocks are
disjoint, Apostol's formula shows that
$w(n,m)\ne0$ only when $n$ and $m$ differ
in exactly one coordinate, say $j$, by a
prime-power ratio. In that case,
multiplicativity of $\varphi$ gives
\[
w(n,m)
=
\left(\prod_{i\ne j}\varphi(n_i)\right)
w_j(n_j,m_j).
\]
Thus the edge weights are exactly those
of the weighted product graph, with no
additional edges.
Thus all hypotheses of
\cref{lem:product-support} hold.  With
$
v_n=\sum_{j\le J}v_j(n_j)$ and $
n=\prod_jn_j$
we obtain
$
\sum_{n\in\mathcal N}c_nv_n=0
$ and
$
\mathscr A(S)\le-\delta_0
$
for every nonempty proper $S\subsetneq\mathcal N$.

\subsubsection{Choosing the primes}

We now choose the primes defining the blocks. As in Venkatesh's construction~\cite{Venkatesh}, the logarithmic gain ultimately comes from Mertens' product theorem. Here we distribute the relevant primes among several blocks. There are three requirements on the choice of primes. First, we need
$$
P_j=\prod_{p\mid N_j}\frac p{p-1}
\le\frac54,
\qquad
\min_{p\mid N_j}\log p
\ge 8(\log q_j+1).
$$
Thus each block should contain only a bounded amount of contribution to the Euler product,
and its companion primes should lie well above its distinguished prime
$q_j$. Second, we want the total degree $C$ to remain comparable to
$\varphi(M)$.  Since
$$
\frac{C}{\varphi(M)}
=
\prod_{j\le J}
\left(1+\frac{P_j}{q_j-1}\right),
$$
and $P_j\le5/4$, it is therefore enough to arrange
$
\sum_j\frac1{q_j}<\infty,
$
which keeps this product bounded uniformly in $J$.  We will achieve
this by choosing the distinguished primes $q_j$ to grow by a fixed
multiplicative factor. Third, we want the prime blocks, together with the distinguished primes $q_j$, to contain essentially every prime up to a cutoff $X$.  Indeed,
$$
\frac{M}{\varphi(M)}
=
\prod_{p\mid M}\frac{p}{p-1},
$$
so the size of $M/\varphi(M)$ is determined by which primes divide $M$.  It will therefore be enough to arrange that all primes in an
interval
$
[Y_1,X']
$
with $\log X'\asymp\log X$ divide $M$.  Mertens' estimate will then
give
$$
\frac{M}{\varphi(M)}\gg\log X.
$$
We will take $X$ to be a small power of $\log d$, so that
$\log X\asymp\log\log d$.

These considerations lead to the following choice. By Bertrand’s postulate, for every $j$ we may choose an odd prime $q_{j+1}$ with $2q_j<q_{j+1}<4q_j.$ In particular we have 
$$
\sum_{j\ge1}\frac1{q_j}
\le
\frac1{q_1}\sum_{j\ge1}2^{-(j-1)}
<\infty.
$$
Next fix
$
\alpha=6/5
$
and choose $Y_1$ sufficiently large.  Define recursively
$
Y_{j+1}=Y_j^\alpha.
$
The reason for this scale is that Mertens' estimate gives

$$
\prod_{Y_j\le p<Y_{j+1}}\frac p{p-1}
\sim
\frac{\log Y_{j+1}}{\log Y_j}
=
\alpha.
$$
Since $\alpha=6/5<5/4$, after enlarging $Y_1$ we may therefore assume
that
\begin{equation}
    \prod_{Y_j\le p<Y_{j+1}}\frac p{p-1}
    \le\frac54
    \label{eq:block-Euler-upper}
\end{equation}
for every $j$.  At the same time, because
$\log q_j=O(j)$ grows linearly whereas
$\log Y_j=\alpha^{j-1}\log Y_1$ grows exponentially with $j$, we may arrange that
\begin{equation}
    \log Y_j\ge8(\log q_j+1)
    \label{eq:Y-vs-q}
\end{equation}
for every $j$. Finally, we choose the cutoff 
$
X=(\log d)^{1/20}.
$
The particular exponent $1/20$ is inessential.  We have
$
\log X=\frac1{20}\log\log d,
$
so this cutoff still gives the desired $\log\log d$ gain.
Let $J$ be the largest index with $Y_{J+1} \le X$. Since $J=O(\log\log X)$ and
$q_j\le q_1 4^{j-1}$, we have
$q_J\le(\log X)^{O(1)}<X$
for sufficiently large $X$.  For $1\le j\le J$, let
$$
\mathcal P_j
=
\{p\text{ prime}:Y_j\le p< Y_{j+1} \}
\setminus\{q_1,\ldots,q_J\},
$$
and put
$
N_j=\prod_{p\in \mathcal P_j}p.
$
Then the prime factors of the integers $q_jN_j$ are pairwise disjoint.
By \eqref{eq:block-Euler-upper} we have 
$$
P_j
=
\prod_{p\in\mathcal P_j}\frac p{p-1}
\le
\prod_{Y_j\le p<Y_{j+1}}\frac p{p-1}
\le\frac54,
$$
and every $p\mid N_j$ satisfies
$$
\log p
\ge
\log Y_j
\ge
8(\log q_j+1).
$$
Thus every block satisfies the hypotheses of
\cref{lem:one-divisor-block}. 

The preceding degree formula and the bounds
$P_j\le5/4$, $\sum_jq_j^{-1}<\infty$ give
$
C\asymp\varphi(M).
$
Moreover, every prime in
$
[Y_1,Y_{J+1})
$
divides $M$.  Since $J$ is maximal with $Y_{J+1}\le X$ we have 
$
Y_{J+1}\le X<Y_{J+2}=Y_{J+1}^{\alpha},
$
and hence
$
\log Y_{J+1}\asymp\log X.
$
Therefore by Mertens' estimate we get 
$$
\begin{aligned}
    \frac MC
    \asymp
    \frac{M}{\varphi(M)}
    \gg
    \prod_{Y_1\le p<Y_{J+1}}\frac p{p-1} 
    \asymp
    \log Y_{J+1}
    \asymp
    \log X
    \asymp
    \log\log d.
\end{aligned}
$$

\subsubsection{Verification of admissibility}

It remains to verify the quantitative conditions in
\cref{subsec:gluing-criterion}.  Since $1\in\mathcal D_j$ for every
$j$, we have $1\in\mathcal N$, while all primes used in the
construction are odd, so $2\notin\mathcal N$. The volume-preserving condition
$
\sum_{n\in\mathcal N}c_ny_n=0
$
and the uniform bound
$
\mathscr A(S)\le-\delta_0
$
for every nonempty proper $S\subsetneq\mathcal N$ were established
above.  It remains to verify the remaining quantitative bounds on $I/C$, $y_1$, $M$, $C$ and
$r$.

We first bound the total resultant weight $I$. Since $M$ is squarefree, every nonzero resultant edge is
$n\leftrightarrow pn$ with weight $c_n\log p$. Hence
$$
I \le \sum_{n\in\mathcal N} c_n\sum_{p\mid M}\log p = C\log M.
$$
Since $M\in\mathcal N$, we have $C\ge\varphi(M)$.  As $M$ is odd and
squarefree we have 
$$
\varphi(M)
=
\prod_{p\mid M}(p-1)
\ge
\prod_{p\mid M}\sqrt p
=
\sqrt M.
$$
Hence $\log M\le2\log C$ and consequently $ \frac IC\ll\log C.$

For the $n=1$ component, recall that
$
y_1=I/C- \omega_1/2+v_1.
$
Here $\omega_1\le\log M$.  Moreover,
$v_1=\sum_{j\le J}v_{1,j}$, and
\eqref{eq:local-potential} together with
$\Theta_j\le\log q_j$ gives
$
|v_{1,j}|\ll1+\log q_j.
$
Since each $q_j$ divides $M$ this gives 
$$
|v_1|
\ll
J+\sum_{j\le J}\log q_j
\ll
\log M.
$$
Thus
$
|y_1|\ll\log M\ll\log C.
$

It remains to compare the arithmetic parameters with $d$.
Every prime divisor of $M$ is at most $X$.  Hence, with
$
\vartheta(X):=\sum_{p\le X}\log p,
$
Chebyshev's estimate gives
$$
\log M\le\vartheta(X)=O(X)=o(\log d).
$$ 
Since every $n\in\mathcal N$ divides $M$ we have
$
C
=
\sum_{n\in\mathcal N}\varphi(n)
\le
\sum_{n\mid M}\varphi(n)
=
M,
$
so
$$
C=\exp(O(X))=d^{o(1)}.
$$
For the required lower bound on $C$, note that every prime in
$[Y_1,Y_{J+1})$ divides $M$.  Hence the lower Chebyshev bound   $ \vartheta(x)\gg x$
gives
$$
\log M
\ge
\sum_{Y_1\le p<Y_{J+1}}\log p
=
\vartheta(Y_{J+1})-\vartheta(Y_1)
\gg
Y_{J+1}, 
$$
Using 
$
X<Y_{J+2}=Y_{J+1}^{\alpha}
$
yields 
$
Y_{J+1}>X^{1/\alpha}.
$
Then $C \ge \sqrt M$ implies that 
$
\log C
\gg
X^{1/\alpha}.
$
With $\alpha=6/5$ and $X=(\log d)^{1/20}$, this gives
$
\log C\gg(\log d)^{1/24}\gg\log\log d,
$
and therefore
$
C\ge\log d
$
for all sufficiently large $d$. Finally, since 
$
r=\left\lfloor\frac dC\right\rfloor
$
and $C=d^{o(1)}$, we have 
$$
\frac{r}{C^3\log d}
\asymp
\frac{d}{C^4\log d}
=
d^{1-o(1)},
$$
and hence
$
r\gg C^3\log d.
$
All the admissibility conditions are therefore satisfied.   Together with
$
\frac MC\gg\log\log d,
$
this proves \cref{thm:arithmetic-realization}.

\section{The stochastic argument}
\label{sec:stochastic}

We now combine the arithmetic construction with the stochastic ellipsoid
method of Klartag and Abuya--Gargava--Zhao
\cite{Klartag,AbuyaGargavaZhao}. We select a lattice with an empty
initial ball and few expected contact orbits, and evolve it within the
$\tau$-invariant quadratic forms.  Since each contact orbit removes at
most one deformation direction, the active space will remain large enough
for the process to produce a lattice-free ellipsoid of large volume.

Throughout this section, we work with arbitrary
admissible gluing data and the associated
normalized Haar ensemble.

\subsection{Invariant quadratic forms}
\label{subsec:invariant-forms}

To preserve the cyclic symmetry throughout the stochastic
deformation, we restrict the process to a subspace of
$\tau$-invariant quadratic forms. Its dimension $q$ is the reservoir of available deformation
directions and will determine the scale of the eventual volume gain.

Recall that
$
d=Cr+a$ with $
0\le a<C.
$ 
The $n=1$ component has real dimension $r+a$. Since
$2\notin\mathcal N$ and $\varphi(n)$ is even for every $n>2$,
the embeddings of each nontrivial cyclotomic component occur in
pairs of complex conjugates. Hence, as real vector spaces,
$
\mathbb R^d
\simeq
\mathbb R^{r+a}
\oplus
\bigoplus_{\sigma}\mathbb C^r,
$
where $\sigma$ runs over the conjugate pairs. Define
\begin{equation}
    \mathcal Q
    =
    \mathbb R\operatorname{Id}_{r+a}
    \oplus
    \bigoplus_{\sigma}
    \operatorname{Herm}(\mathbb C^r)
    \subset
    \operatorname{Sym}(\mathbb R^d) =
    \{A\in M_d(\mathbb R):A^T=A\}.
    \label{eq:Qspace}
\end{equation}
Here a Hermitian matrix  $H\in\operatorname{Herm}(\mathbb C^r)$ is viewed as the real quadratic form
$
z\longmapsto z^*Hz
$
on $\mathbb C^r\simeq\mathbb R^{2r}$. The space $\mathcal Q$ consists of $\tau$-invariant quadratic
forms. Indeed, on the summand corresponding to
$\sigma:K_n\hookrightarrow\mathbb C$, the automorphism $\tau$ acts
by multiplication by $\sigma(\zeta_n)$, and therefore
$$
(\sigma(\zeta_n)z)^*H(\sigma(\zeta_n)z)
=
z^*Hz.
$$
Thus, for every $A\in\mathcal Q$ we have 
$
\langle A\tau x,\tau x\rangle
=
\langle Ax,x\rangle.
$
Moreover, $\mathcal Q$ contains $\operatorname{Id}_d$
and is closed under inversion of positive definite
elements. Since
$
\dim_{\mathbb R}\operatorname{Herm}(\mathbb C^r)=r^2
$
and there are $(C-1)/2$ conjugate pairs,
\begin{equation}
    q:=\dim_{\mathbb R}\mathcal Q
    =
    1+\frac{C-1}{2}r^2.
    \label{eq:q}
\end{equation}
Since $d=Cr+a$ with $0\le a<C$, we have $
q\asymp Cr^2\asymp dr.
$

We record two consequences of the admissibility conditions that will be
used below. Since $q\ll dr\le d^2$ and
$\log M\ll\log d$, we have 
\begin{equation}
    \log(Mq)=O(\log d).
    \label{eq:logMq}
\end{equation}
Furthermore,
$
r\gg C^3\log d
$ and $
C\ge\log d 
$
imply
\begin{equation}
    \frac{(\log d)^{3/2}}{\sqrt r}
    \ll
    \frac{\log d}{C^{3/2}}
    \longrightarrow0.
    \label{eq:R-stochastic-error}
\end{equation}

\subsection{An angular projection estimate}
\label{subsec:angular}

The contact probability depends on how strongly a lattice direction
interacts with the allowed deformation space $\mathcal Q$. We therefore
need an angular estimate for random directions on the sphere.

Let $\Pi:\operatorname{Sym}(\mathbb R^d)\longrightarrow\mathcal Q $ denote orthogonal projection with respect to the Hilbert--Schmidt inner product, and let $B_1,\ldots,B_q$ be an orthonormal basis of
$\mathcal Q$. For $\omega\in S^{d-1}$ we have 
$
\Pi(\omega\otimes\omega)
=
\sum_{j=1}^q
\langle B_j\omega,\omega\rangle B_j,
$
and hence
$$
\bigl\|\Pi(\omega\otimes\omega)\bigr\|_{\mathrm{HS}}^2
=
\sum_{j=1}^q
\langle B_j\omega,\omega\rangle^2.
$$

\begin{lemma}
    \label{lem:angular}
    Let $\omega$ be uniformly distributed on $S^{d-1}$.  Then
    $$
    \mathbb E
    \bigl\|
    \Pi(\omega\otimes\omega)
    \bigr\|_{\mathrm{HS}}^2
    =
    \frac{2q+d}{d(d+2)}.
    $$
    Moreover, there are absolute constants $c,K>0$ such that, whenever
    $$
    0\le\lambda
    \le
    c\sqrt{\frac{2q+d}{d}},
    $$
    we have
    $$
    \mathbb E
    \exp\left(
        \lambda
        \frac{d(d+2)}{2q+d}
        \bigl\|
        \Pi(\omega\otimes\omega)
        \bigr\|_{\mathrm{HS}}^2
    \right)
    \le
    K e^\lambda.
    $$

\end{lemma}

\begin{proof}
    For any real symmetric matrix $B$, let
    $\lambda_1,\ldots,\lambda_d$ be its eigenvalues. By rotational
    invariance, we may assume
    $B=\operatorname{diag}(\lambda_1,\ldots,\lambda_d)$. Using
    $$
    \mathbb E\omega_i^4=\frac{3}{d(d+2)},
    \qquad
    \mathbb E(\omega_i^2\omega_j^2)=\frac{1}{d(d+2)}
    \quad (i\ne j),
    $$
    we obtain
    $$
    \begin{aligned}
        \mathbb E\langle B\omega,\omega\rangle^2
        =
        \frac{
            3\sum_i\lambda_i^2
            +2\sum_{i<j}\lambda_i\lambda_j
        }{d(d+2)}
        =
        \frac{
            2\operatorname{Tr}(B^2)
            +\operatorname{Tr}(B)^2
        }{d(d+2)}.
    \end{aligned}
    $$
    Summing the preceding identity gives
    $$
    \mathbb E
    \bigl\|
    \Pi(\omega\otimes\omega)
    \bigr\|_{\mathrm{HS}}^2
    = 
    \frac{
        2q+
        \sum_{j=1}^q\operatorname{Tr}(B_j)^2
    }{d(d+2)}.
    $$
    Since $\operatorname{Tr}(B_j)=\langle B_j,\operatorname{Id}_d\rangle_{\mathrm{HS}}$
    and $\operatorname{Id}_d\in\mathcal Q$, Parseval gives
    $
    \sum_{j=1}^q \operatorname{Tr}(B_j)^2
    =
    \|\operatorname{Id}_d\|_{\mathrm{HS}}^2
    =
    d,
    $
    which proves the first assertion.

    Set
    $$
    F(\omega)
    =
    \sqrt{
        \frac{d(d+2)}{2q+d}
    }\,
    \bigl\|
    \Pi(\omega\otimes\omega)
    \bigr\|_{\mathrm{HS}}.
    $$
    Then $\mathbb EF^2=1.$ Moreover, for $u,v\in S^{d-1}$ we have  $ \|u\otimes u-v\otimes v\|_{\mathrm{HS}} \le 2\|u-v\|.$ Since orthogonal projection is contractive in Hilbert--Schmidt norm, it follows
    that $F(\omega)$ is $L$-Lipschitz on $S^{d-1}$, with Lipschitz constant 
    $$
    L
    \le
    2\sqrt{\frac{d(d+2)}{2q+d}}.
    $$

    By the concentration inequality for Lipschitz functions on the sphere
    (see, for example, \cite[Theorem~5.1.3]{VershyninHDP}), there is an absolute constant
    $c_0>0$ such that
    $$
    \Pr\{|F-\mathbb EF|>s\}
    \le
    2\exp\left(-c_0\frac{d}{L^2}s^2\right),
    \qquad s\ge0.
    $$
    Since $\mathbb EF^2=1$, Cauchy--Schwarz gives $\mathbb EF\le1.$ Since 
    $
    \frac{d}{L^2}
    \ge
    \frac{2q+d}{4(d+2)}
    \asymp
    \frac{2q+d}{d} =: \alpha
    $
    we obtain, after adjusting the constant $c_0$,
    $$
    \Pr\{F>1+s\}
    \le
    2e^{-c_0\alpha s^2},
    \qquad s\ge0.
    $$
    Choose $c_1>0$ sufficiently small, and suppose that $0\le\lambda\le c_1\sqrt{\alpha}.$ Since $\alpha\ge1$, after decreasing $c_1$ if necessary, we have
    $\lambda\le c_0\alpha/4$. Moreover, by Young's inequality,
    $$
    2\lambda s
    \le
    \frac{c_0\alpha}{4}s^2
    +
    \frac{4\lambda^2}{c_0\alpha}.
    $$
    Hence, for $s\ge0$, we have 
    $$
    \begin{aligned}
        \lambda(1+s)^2-c_0\alpha s^2
        =
        \lambda+2\lambda s+(\lambda-c_0\alpha)s^2
        \le
        \lambda
        +
        \frac{4\lambda^2}{c_0\alpha}
        -
        \frac{c_0\alpha}{2}s^2.
    \end{aligned}
    $$
    For $\lambda>0$, integration by parts gives
    $$
    \mathbb E e^{\lambda F^2}
    \le
    e^\lambda
    +
    2\lambda
    \int_1^\infty
    t e^{\lambda t^2}\Pr\{F>t\}\,dt.
    $$
    Writing $t=1+s$ and using the preceding tail estimate, we obtain
    $$
    \begin{aligned}
        \mathbb E e^{\lambda F^2}
    &\le
    e^\lambda
    +
    4\lambda
    \int_0^\infty
    (1+s)
    e^{\lambda(1+s)^2-c_0\alpha s^2}\,ds
    \\
    &\le
    e^\lambda
    \left[
        1+
        4e^{4\lambda^2/(c_0\alpha)}
        \lambda
        \int_0^\infty
        (1+s)e^{-c_0\alpha s^2/2}\,ds
    \right].
    \end{aligned}
    $$
    Since
    $$
    \int_0^\infty
    (1+s)e^{-c_0\alpha s^2/2}\,ds
    =
    O(\alpha^{-1/2}+\alpha^{-1})
    =
    O(\alpha^{-1/2}),
    $$
    and $\frac{\lambda^2}{\alpha}\le c_1^2$, the quantity in brackets is bounded by a constant. Thus
    $
    \mathbb E e^{\lambda F^2}
    \le
    K e^\lambda
    $
    for some absolute $K>0$. This concludes the proof.
\end{proof}

\subsection{The stochastic estimate with weights given by orbit size}
\label{subsec:stochastic-estimate}
Because every quadratic form in $\mathcal Q$ is $\tau$-invariant, all
vectors in a single $\tau$-orbit impose the same linear contact
condition. Thus the number of deformation directions lost to a contact
set is controlled by the number of contact orbits rather than by the
number of contact vectors. This is recorded by assigning a nonzero
vector $x$ the weight $1/M_{\operatorname{supp}(x)}$, the reciprocal
of its orbit size. We now adapt the Klartag--Abuya--Gargava--Zhao stochastic estimate to
these variable orbit weights.

Recall that $\Lambda_g^*=\lambda T_y gL.$
Here $g$ acts separately and $K_n$-linearly on each cyclotomic
component $V_n$, while $T_y$ acts by a scalar on each $V_n$.
The final factor $\lambda$ is a common scalar. Hence all three maps
preserve cyclotomic support and commute with $\tau$. Consequently the
formula for orbit size in \cref{lem:orbit} remains valid on
$\Lambda_g^*$.
In particular, $
M_{\operatorname{supp}(x)}=M $
for every vector of full cyclotomic support, while vectors of proper support
may have smaller orbits. 

We now introduce a stochastic deformation incorporating these orbit weights. Recall that $v_d$ is the volume of the Euclidean unit ball in $\mathbb R^d$. Fix a $\tau$-invariant lattice
$
\Lambda\subset\mathbb R^d
$
of covolume $v_d$.  For a positive definite quadratic form $A\in\mathcal Q$, viewed as a symmetric
operator on $\mathbb R^d$, let 
$
E_A
=
\{x\in\mathbb R^d:\langle Ax,x\rangle<1\}
$
be the associated open ellipsoid. Since $A$ is positive definite, we have 
$$
E_A=A^{-1/2}B_1,
\qquad
\vol(E_A)=v_d(\det A)^{-1/2}.
$$
If
$
E_A\cap\Lambda=\{0\},
$
define
$$
F_A
=
\left\{
    B\in\mathcal Q:
    \langle Bx,x\rangle=0
    \text{ for every }
    x\in\partial E_A\cap\Lambda
\right\}.
$$
Thus $F_A$ is the linear space of deformation
directions that preserve all contacts because if $B\in F_A$ and $x\in\partial E_A\cap\Lambda$, then
$
\langle (A+sB)x,x\rangle=1
$
for every $s$ for which $A+sB$ remains positive definite.

We use the process constructed in
\cite[Proposition~2.3 and Corollary~2.4]{Klartag}. This process preserves
all contacts, and we restrict its Brownian motion to the fixed Euclidean
space $\mathcal Q$. See also
\cite[Proposition~3.2]{AbuyaGargavaZhao}.  Write $P_A$ for orthogonal
projection onto $F_A$.  Starting from a lattice-free scalar matrix
$A_0=a_0\operatorname{Id}$, the process satisfies
$$
dA_t=P_{A_t}\,dW_t,
$$
where $W_t$ is standard Hilbert--Schmidt Brownian motion in
$\mathcal Q$.  It is constant after the active space becomes zero.
It stays positive definite and lattice-free, and preserves all contacts. For $T>0$, $A_T$ denotes the value of the stochastic deformation
at time $T$, and $E_{A_T}$ is the corresponding terminal ellipsoid.  

We now introduce the quantities needed to bound the probability that a
fixed lattice vector becomes a contact by time $t$, and hence to
control the expected weighted number of contacts. We write $\Pr_{\mathrm{Br}}$ and $\mathbb E_{\mathrm{Br}}$ for
probability and expectation with respect to the Brownian motion driving
the process, and set
$$
\Psi(s)=\Pr\{Z\ge s\},
\qquad Z\sim N(0,1).
$$
Put
$
a_0=(1-1/d)^{-2}.
$
Since $A_0=a_0\operatorname{Id}$, the initial ellipsoid is the ball 
$
E_{A_0}=B_{1-1/d}.
$
Thus $B_{1-1/d}\cap\Lambda=\{0\}$ means precisely that the process
starts from a lattice-free ellipsoid. For $0\ne x\in\mathbb R^d$, define
$$
\beta(x)^2
=
\frac{\|\Pi(x\otimes x)\|_{\mathrm{HS}}^2}{\|x\|^4}.
$$
Since $\operatorname{Id}_d/\sqrt d\in\mathcal Q$, we have $\beta(x)^2\ge\frac1d,$ and hence $\beta(x)>0$. 

\begin{lemma}
    \label{lem:stochastic-prelim}
    There is an absolute constant $C_{\mathrm{op}}>0$ such that, for
    $t>0$,
    $$
    \Pr_{\mathrm{Br}}\left\{
        \|A_t-a_0\operatorname{Id}\|_{\mathrm{op}}
        >
        C_{\mathrm{op}}\sqrt{dt}
    \right\}
    \le e^{-2d}.
    $$
    Moreover, for every $0\ne x\in\Lambda$ we have 
    $$
    \Pr_{\mathrm{Br}}\left\{
        x \text{ becomes a contact by time }t
    \right\}
    \le
    2\Psi\left(
        \frac{a_0-\|x\|^{-2}}
        {\beta(x)\sqrt t}
    \right).
    $$
\end{lemma}
\begin{proof}
    This is \cite[Lemmas~3.4 and~3.6]{AbuyaGargavaZhao}.
\end{proof}
Define the good event 
$
\mathcal G_t
=
\left\{
    \|A_t-a_0\operatorname{Id}\|_{\mathrm{op}}
    \le
    C_{\mathrm{op}}\sqrt{dt}
\right\}.
$
Set
$$
\rho_t
=
\frac1{\sqrt{a_0-C_{\mathrm{op}}\sqrt{dt}}},
$$
whenever the denominator is positive.  On $\mathcal G_t$, every
contact vector satisfies
$
1-\frac1d\le\|x\|\le\rho_t.
$
Here the lower bound follows from the initial lattice-free ball, while
for $x\in\partial E_{A_t}\cap\Lambda$,
$$
1=\langle A_tx,x\rangle
\ge
\bigl(a_0-C_{\mathrm{op}}\sqrt{dt}\bigr)\|x\|^2.
$$
Accordingly, define
\begin{equation}
    K_t(\Lambda)
    =
    \sum_{0\ne x\in\Lambda}
    \frac1{M_{\operatorname{supp}(x)}}
    \mathbf 1_{\{1-1/d\le\|x\|\le\rho_t\}}
    \Psi\left(
        \frac{a_0-\|x\|^{-2}}
        {\beta(x)\sqrt t}
    \right).
    \label{eq:Kt}
\end{equation}
The quantity $K_t(\Lambda)$ will be used to bound the expected contact
count at time $t$, with each vector weighted by the reciprocal of its
orbit size.

\begin{proposition}[Klartag--AGZ estimate with weights given by orbit size]
    \label{prop:stochastic}
    Suppose
    $$
    B_{1-1/d}\cap\Lambda=\{0\}.
    $$
    Then
    $$
    \dim F_{A_t}
    \ge
    q-
    \sum_{x\in\partial E_{A_t}\cap\Lambda}
    \frac1{M_{\operatorname{supp}(x)}},
    $$
    and
    $$
    \mathbb E_{\mathrm{Br}}
    \sum_{x\in\partial E_{A_t}\cap\Lambda}
    \frac1{M_{\operatorname{supp}(x)}}
    \le
    2K_t(\Lambda)+e^{-cd}.
    $$
    Moreover, whenever $dT\le c$,
    \begin{align}
        \mathbb E_{\mathrm{Br}}\log\det A_T
        \le{}&
        -\frac{qT}{2}
        +
        \frac12
        \int_0^T
        \mathbb E_{\mathrm{Br}}
        \sum_{x\in\partial E_{A_t}\cap\Lambda}
        \frac1{M_{\operatorname{supp}(x)}}\,dt
        \notag\\
             &+
             O\left(
                 1+\frac{qT}{d}
                 +qT\sqrt{dT}
                 +qTe^{-2d}
             \right).
             \label{eq:stochastic-det}
    \end{align}
\end{proposition}

\begin{proof}
    The dimension bound is an analogue of
    \cite[Lemma~3.3]{AbuyaGargavaZhao} that allows different orbit sizes. Since every $B\in\mathcal Q$ is $\tau$-invariant, we have $\langle B\tau x,\tau x\rangle=\langle Bx,x\rangle.$ In particular, if $x\in\partial E_{A_t}\cap\Lambda$, then
    $\tau x\in\Lambda$ and
    $$
    \langle A_t\tau x,\tau x\rangle
    =
    \langle A_tx,x\rangle
    =
    1.
    $$
    Therefore the contact set $\partial E_{A_t}\cap\Lambda$  is a union of complete $\tau$-orbits and hence 
    $$
    \left|\{\text{contact orbits}\}\right|
    =
    \sum_{x\in\partial E_{A_t}\cap\Lambda}\frac1{M_{\operatorname{supp}(x)}}.
    $$
    Each contact orbit imposes at most one independent linear condition on
    $F_{A_t}$. Hence
    $$
    \dim F_{A_t}
    \ge
    q-
    \sum_{x\in\partial E_{A_t}\cap\Lambda}\frac1{M_{\operatorname{supp}(x)}}.
    $$

    The contact estimate is proved as in
    \cite[Lemma~3.7]{AbuyaGargavaZhao}, with the constant orbit weight
    $1/m$ there replaced by $1/M_{\operatorname{supp}(x)}$.  \cref{lem:stochastic-prelim} gives
    $$
    \begin{aligned}
        \mathbb E_{\mathrm{Br}}\left[
            \mathbf1_{\mathcal G_t}
            \sum_{x\in\partial E_{A_t}\cap\Lambda}\frac1{M_{\operatorname{supp}(x)}}
        \right]
        \le
        2
        \sum_{0\ne x\in\Lambda}
        \frac1{M_{\operatorname{supp}(x)}}
        \mathbf1_{\{1-1/d\le\|x\|\le\rho_t\}}
        \Psi\left(
            \frac{a_0-\|x\|^{-2}}
            {\beta(x)\sqrt t}
        \right)
        =
        2K_t(\Lambda).
    \end{aligned}
    $$
    On $\mathcal G_t^c$, the boundary bound
    \cite[Lemma~3.1]{AbuyaGargavaZhao},
    $
    \left|\partial E_{A_t}\cap\Lambda\right|
    \le 2(2^d-1),
    $
    together with
    $\Pr_{\mathrm{Br}}(\mathcal G_t^c)\le e^{-2d}$,
    gives a contribution $e^{-cd}$.   Hence
    $$
    \mathbb E_{\mathrm{Br}}
    \sum_{x\in\partial E_{A_t}\cap\Lambda}\frac1{M_{\operatorname{supp}(x)}}
    \le
    2K_t(\Lambda)+e^{-cd}.
    $$

    Finally, we follow the proof of
    \cite[Proposition~3.5]{AbuyaGargavaZhao}.  In our notation, their
    It\^o calculation gives
    $$
    \mathbb E_{\mathrm{Br}}\log\det A_T
    \le
    -\frac12
    \int_0^T
    \mathbb E_{\mathrm{Br}}\dim F_{A_t}\,dt
    +
    O\left(
        1+\frac{qT}{d}
        +qT\sqrt{dT}
        +qTe^{-2d}
    \right),
    $$
    whenever $dT\le c$.  Substituting the dimension bound above yields
    $$
    \begin{aligned}
        \mathbb E_{\mathrm{Br}}\log\det A_T
        \le{}&
        -\frac{qT}{2}
        +
        \frac12
        \int_0^T
        \mathbb E_{\mathrm{Br}}
        \sum_{x\in\partial E_{A_t}\cap\Lambda}
        \frac1{M_{\operatorname{supp}(x)}}\,dt
        \\
             &+
             O\left(
                 1+\frac{qT}{d}
                 +qT\sqrt{dT}
                 +qTe^{-2d}
             \right),
    \end{aligned}
    $$
    as claimed.
\end{proof}

\subsection{Bounding the expected number of contact orbits}
\label{subsec:contact}

We now estimate $K_t$ on average over the normalized Haar ensemble and choose
the terminal time so that $\int_0^T\mathbb E_g K_t\,dt$ remains bounded.

Let $\Lambda_g^*$ denote the normalized Haar ensemble constructed in
\cref{subsec:haar-averaging,subsec:supportwise-mean}, with $g\in X_L$
distributed according to Haar probability measure.
Every $\Lambda_g^*$ has covolume $v_d$ and is
invariant under $\tau$. Set
\begin{equation}
    \gamma
    =
    \frac{d(2q+d)}{8(d+2)}.
    \label{eq:gamma}
\end{equation}
Since $q\asymp dr$ and $r\to\infty$ by admissibility, we have $d=o(q)$. Hence $
\gamma
\asymp 2q + d \asymp q.
$
Write
$$
K_t(\Lambda_g^*)
=
K_t^{\mathrm{full}}(\Lambda_g^*)
+
K_t^{\mathrm{proper}}(\Lambda_g^*)
$$
according to whether the lattice vector has full or proper cyclotomic
support. Let  $K_0>0$ be a constant, to be chosen sufficiently large
below, and define the terminal time
\begin{equation}
    T
    =
    \gamma^{-1}
    \log\left(\frac{M\gamma}{K_0}\right).
    \label{eq:T}
\end{equation}
For all sufficiently large $d$, this is positive.  Moreover, by the
admissibility conditions, 
$$
\gamma T=O(\log d),
\qquad
dT=O\left(\frac{\log d}{r}\right)=o(1).
$$

\begin{lemma}[Contribution from vectors of full support]
    \label{lem:full-contact}
    There exists a constant $K>0$ such that, for every
    $0<t\le T$ we have 
    $$
    \mathbb E_g K_t^{\mathrm{full}}(\Lambda_g^*)
    \le
    \frac{K}{M}e^{\gamma t}.
    $$
\end{lemma}

\begin{proof}
    Every vector of full support has orbit size $M$.  The Haar
    mean-value formula for vectors of prescribed support and the normalization
    $\covol(\Lambda_g^*)=v_d$ give  
    \begin{align*}
        \mathbb E_g
        K_t^{\mathrm{full}}(\Lambda_g^*)
        =
        \frac1{Mv_d}
        \int_{\mathbb R^d}
        \mathbf1_{\{
                1-1/d\le\|x\|\le\rho_t
        \}}
        \Psi\left(
            \frac{
                a_0-\|x\|^{-2}
            }{
                \beta(x)\sqrt t
            }
        \right)\,dx 
        = \frac1M
        \int_{S^{d-1}}
        J_t(\omega)\,d\sigma(\omega),
    \end{align*}
    where $d\sigma$ denotes the probability measure on $S^{d-1}$ and 
    $$
    J_t(\omega)
    =
    d
    \int_{1-1/d}^{\rho_t}
    \Psi\left(
        \frac{
            a_0-r^{-2}
        }{
            \beta(\omega)\sqrt t
        }
    \right)
    r^{d-1}\,dr.
    $$
    The radial estimate from \cite[Lemma~4.1]{AbuyaGargavaZhao} implies that
    $$
    J_t(\omega)\le
    \exp\left((1+K_1\sqrt{dt})\frac{d^2t}{8}\beta(\omega)^2\right)
    $$
    for an absolute constant $K_1>0$. Set
    $$
    F(\omega)^2
    =
    \frac{d(d+2)}{2q+d}\,\beta(\omega)^2.
    $$
    Then
    $$
    \frac{d^2t}{8}\beta(\omega)^2
    =
    \gamma t F(\omega)^2.
    $$   
    Thus, with
    $
    \lambda_t
    =
    (1+K_1\sqrt{dt})\gamma t,
    $
    we have
    $
    J_t(\omega)\le e^{\lambda_tF(\omega)^2}.
    $
    By admissibility and the estimates preceding the lemma,
    $$
    \lambda_t=O(\log d),
    \qquad
    \sqrt{\frac{2q+d}{d}}\asymp\sqrt r,
    $$
    and $\log d=o(\sqrt r)$.  Hence $\lambda_t$ lies in the range of
    \cref{lem:angular}, uniformly for $0<t\le T$.  Therefore
    $$
    \int_{S^{d-1}}J_t(\omega)\,d\sigma(\omega)
    \le
    K_2e^{\lambda_t}.
    $$
    Moreover, by
    $
    \lambda_t-\gamma t
    \le
    K_1\gamma T\sqrt{dT}
    =
    o(1),
    $
    we obtain 
    $$
    \int_{S^{d-1}}J_t(\omega)\,d\sigma(\omega)
    \le
    Ke^{\gamma t}.
    $$
    This proves the result.
\end{proof}

\begin{lemma}[Contribution from vectors of proper support]
    \label{lem:proper-contact}
    There exists a constant $c>0$ such that for every
    $0<t\le T$ we have 
    $$
    \mathbb E_g
    K_t^{\mathrm{proper}}(\Lambda_g^*)
    \le
    e^{-cr}.
    $$
\end{lemma}

\begin{proof}
    Uniformly for $0<t\le T$ we have 
    \begin{equation*}
        \log\rho_t
        =
        O\left(
            d^{-1}
            +
            \sqrt{\frac{\log d}{r}}
        \right).
        \label{eq:rho-stochastic}
    \end{equation*}
    The estimate for vectors of proper support in \cref{cor:proper-supports} is therefore
    applicable. Since $ 1 / M_{\operatorname{supp}(x)} \le1$ and $0\le\Psi(\cdot)\le1,$ their contribution to $K_t$ is bounded by the number of lattice vectors of proper support in the ball $B_{\rho_t}$. The claimed bound therefore follows
    from \cref{cor:proper-supports}. 
\end{proof}

Combining the estimates for vectors of full and proper support from
\cref{lem:full-contact,lem:proper-contact}, we obtain
\begin{equation*}
    \mathbb E_g K_t(\Lambda_g^*)
    \le
    \frac{K}{M}e^{\gamma t}
    +
    e^{-cr},
    \qquad
    0<t\le T.
    \label{eq:Kt-average}
\end{equation*}
Since the integrand is nonnegative, Tonelli's theorem gives
$$
\begin{aligned}
    \mathbb E_g
    \int_0^T
    K_t(\Lambda_g^*)\,dt
    =
    \int_0^T
    \mathbb E_g
    K_t(\Lambda_g^*)\,dt
    \le
    \int_0^T
    \left(
        \frac{K}{M}e^{\gamma t}
        +e^{-cr}
    \right)\,dt.
\end{aligned}
$$
The term from vectors of full support satisfies
$$
\int_0^T
\frac{K}{M}e^{\gamma t}\,dt
\le
\frac{K}{M\gamma}e^{\gamma T}
=
\frac{K}{K_0},
$$
while
$
Te^{-cr}=o(1).
$
Choosing $K_0$ sufficiently large therefore gives
\begin{equation}
    \mathbb E_g
    \int_0^T
    K_t(\Lambda_g^*)\,dt
    <
    \frac1{100}
    \label{eq:K-integral}
\end{equation}
for all sufficiently large $d$.

\subsection{Selection of a good lattice}
\label{subsec:stochastic-selection}

We now select a single lattice satisfying the two properties required by
the stochastic argument. The initial ball contains no nonzero lattice
points, and the integral of $K_t$ is small. 

For a lattice $\Lambda$ in the normalized Haar ensemble, let 
$$
N_0(\Lambda)
=
\left|\left(
    (\Lambda\setminus\{0\})
    \cap B_{1-1/d}
\right)\right|
$$
be the number of nonzero lattice points in the initial ball. Decompose
$
N_0
=
N_0^{\mathrm{full}}
+
N_0^{\mathrm{proper}}
$
according to cyclotomic support. For vectors of full support, applying the mean-value formula for vectors of this support
\eqref{eq:supportwise-siegel} with
$S=\mathcal N$ and using
$\covol(\Lambda_g^*)=v_d$, we obtain
$$
\mathbb E_g
N_0^{\mathrm{full}}(\Lambda_g^*)
=
\frac{\vol(B_{1-1/d})}{v_d}
=
\left(1-\frac1d\right)^d.
$$
For vectors of proper support, \cref{cor:proper-supports} gives
$
\mathbb E_g
N_0^{\mathrm{proper}}(\Lambda_g^*)
\le
e^{-cr}.
$
Hence
\begin{equation}
    \mathbb E_g
    N_0(\Lambda_g^*)
    \le
    \left(1-\frac1d\right)^d
    +e^{-cr}
    =
    e^{-1}+o(1).
    \label{eq:inner-ball-average}
\end{equation}
By \eqref{eq:K-integral} and Markov's inequality we get 
$$
\Pr_g\left\{
    \int_0^T
    K_t(\Lambda_g^*)\,dt
    >
    \frac1{10}
\right\}
<
\frac1{10}.
$$
Similarly,
$$
\Pr_g\{
    N_0(\Lambda_g^*)\ge1
\}
\le
\mathbb E_g
N_0(\Lambda_g^*)
\le
e^{-1}+o(1).
$$
Since
$
\frac1{10}+e^{-1}<1,
$
for all sufficiently large $d$ the two exceptional events have total
Haar probability strictly less than one.  Hence there exists
$g\in X_L$ such that the lattice
$
\Lambda:=\Lambda_g^*
$
has covolume $v_d$, is $\tau$-invariant, and satisfies
\begin{equation}
    B_{1-1/d}\cap\Lambda=\{0\},
    \qquad
    \int_0^T K_t(\Lambda)\,dt
    \le
    \frac1{10}.
    \label{eq:selected-K}
\end{equation} 

\subsection{The general packing bound}
\label{subsec:general-packing-bound}

We now apply the stochastic estimate with weights $1/M_{\operatorname{supp}(x)}$ to the lattice
selected above. The small value of $\int_0^T K_t(\Lambda)\,dt$ leaves essentially all
$q$ invariant deformation directions available, allowing the stochastic
deformation to produce a large lattice-free ellipsoid. 

\begin{proposition}
    \label{prop:packing}
    For every sufficiently large $d$ for which admissible gluing data exist,
    there is a lattice packing in $\mathbb R^d$ with density
    $$
    \Delta_d^L
    \gg
    d^2\frac{M}{C}\,2^{-d}.
    $$

\end{proposition}

\begin{proof}
    Apply \cref{prop:stochastic} to the lattice $\Lambda$ selected in 
    \cref{subsec:stochastic-selection}. Using \eqref{eq:selected-K} gives 
    $$
    \int_0^T
    \mathbb E_{\mathrm{Br}}
    \sum_{x\in\partial E_{A_t}\cap\Lambda}
    \frac1{M_{\operatorname{supp}(x)}}\,dt
    =
    O(1).
    $$
    Note that 
    $
    \frac{qT}{d}=o(1),
    $
    and, by \eqref{eq:R-stochastic-error},
    $$
    qT\sqrt{dT}
    =
    O\left(
        \frac{(\log d)^{3/2}}{\sqrt r}
    \right)
    =
    o(1).
    $$
    The remaining error term $  qTe^{-2d}$ in \eqref{eq:stochastic-det} 
    is also negligible.  Therefore \eqref{eq:stochastic-det} becomes 
    \begin{equation*}
        \mathbb E_{\mathrm{Br}}\log\det A_T
        \le
        -\frac{qT}{2}
        +O(1).
        \label{eq:det-final}
    \end{equation*}
    Hence there exists a realization for which the final ellipsoid
    $
    E=E_{A_T}
    $
    is lattice-free and
    $$
    \log\det A_T
    \le
    -\frac{qT}{2}+O(1).
    $$
    For this realization,
    $$
    \frac{\vol(E_{A_T})}{v_d}
    =
    (\det A_T)^{-1/2}
    \gg
    \exp\left(\frac{qT}{4}\right) =     \left(
        \frac{M\gamma}{K_0}
    \right)^{q/(4\gamma)}.
    $$
    Note that 
    $$
    \frac{q}{4\gamma}
    =
    \frac{
        2q(d+2)
    }{
        d(2q+d)
    }
    =
    1-O(r^{-1}),
    $$
    where we used $q\asymp dr$.  Moreover,
    $
    \log(M\gamma)
    =
    O(\log d)
    =
    o(r),
    $
    where the last estimate follows from the admissibility assumption $r \gg (\log d)^4 $. It follows that
    $
    (M\gamma/K_0)^{q/(4\gamma)}
    \gg
    M\gamma.
    $
    Since
    $
    \gamma\asymp q,
    $
    we conclude that
    \begin{equation}
        \vol(E)
        \gg
        Mq\,v_d.
        \label{eq:ellipsoid-volume}
    \end{equation}
    Finally, \eqref{eq:q} gives
    $
    q\gg Cr^2.
    $
    Since
    $
    d=Cr+a
    \le
    C(r+1)
    \le
    2Cr
    $
    for all sufficiently large $d$, we have 
    $
    Cr^2
    \gg
    \frac{d^2}{C}.
    $
    Consequently
    \begin{equation}
        Mq
        \gg
        d^2\frac{M}{C}.
        \label{eq:Mq-final}
    \end{equation}
    Because $E$ is lattice-free, the translates of $E/2$ by $\Lambda$
    have disjoint interiors.  Their packing density is
    $$
    \frac{
        \vol(E)
    }{
        2^d\covol(\Lambda)
    }
    =
    \frac{
        \vol(E)
    }{
        2^dv_d
    }.
    $$
    Using \eqref{eq:ellipsoid-volume} and \eqref{eq:Mq-final}, we obtain
    $
    \Delta_d^L
    \gg
    d^2\frac{M}{C}\,2^{-d}.
    $
    A linear change of variables sends the ellipsoid to a Euclidean ball while
    preserving packing density.  This proves the proposition.
\end{proof}

\subsection{Proof of the main theorem}
\label{subsec:proof-main}

We apply Proposition~\ref{prop:packing} to the admissible data supplied by
Theorem~\ref{thm:arithmetic-realization}. Its estimate $M/C\gg\log\log d$
gives
$$
\Delta_d^L\gg d^2\log\log d\,2^{-d},
$$
which proves Theorem~\ref{thm:main}.

\section*{AI-assisted development}
With the assistance of ChatGPT, Simkin pursued the idea of combining Klartag's stochastic framework with the cyclotomic symmetry mechanism of AGZ, with the objective of extending the logarithmic gain to every dimension. This produced a proof of a weaker form of the main result of this paper. 
During the subsequent collaboration with Nihar Gargava and Matthias Gr\"obner, the AI-assisted
exploration was further developed into the stronger result proved in this paper. Nihar Gargava and Matthias Gr\"obner examined the mathematical steps, identified points requiring modification or clarification, and helped reformulate the proof into its present form. The exposition was developed collaboratively, with substantial human rewriting, reorganization, and the addition of motivations, explanations, and intermediate arguments. The authors take responsibility for the mathematical claims and proofs
presented in the paper. AI tools were used with human supervision to assist in the writing.

ChatGPT also played a role in initiating the collaboration by suggesting that Simkin contact Gargava, and later suggesting to Gargava that he contact Gröbner.

\section*{Funding}
Nihar Gargava acknowledges support from the Swiss National Science
Foundation (SNSF), grant 225437, ``Random Geometry with Arithmetic
Constraints.'' Matthias Gr\"obner acknowledges support from SNSF grant
215337, ``Sphere packing, Energy minimization and Fourier interpolation.''

\appendix
\section{A lower bound without using stochastic methods}
\label{app:deterministic-packing}

The gluing construction also gives a lower bound in every dimension of order
$d\log\log d\,2^{-d}$ using Siegel's mean-value theorem and a deterministic
stretching argument.  This appendix proves that bound using only the
arithmetic and mean-value results of
\cref{sec:gluing-framework,sec:arithmetic-realizations}.
The geometric step adapts Minkowski's lemma on successive
minima~\cite{MinkowskiGeometry1910} to the cyclic action and bounds
$\log\det T$ by $H_\rho(\Lambda)$, defined below. A
related version appears in \cite[Proposition~4.1]{GargavaSerbanStretching}.
This approach to the packing bound using successive minima is originally due to Rogers~\cite{RogersExistence1947}.

\begin{theorem}
    \label{thm:deterministic-packing}
    Let $\bigl(\mathcal N,(y_n)_{n\in\mathcal N}\bigr)$
    be any admissible family of gluing data,
    with associated parameters $C$ and $M$.
    Then, as $d\to\infty$,
    \[
        \Delta_d^L
        \ge
        (1-o(1))\frac{dM}{eC}\,2^{-d}.
    \]
    In particular, for the admissible family
    constructed in
    \cref{thm:arithmetic-realization},
    we obtain
    \[
        \Delta_d^L
        \gg
        d\log\log d\,2^{-d}
    \]
    for every sufficiently large integer $d$.
\end{theorem}

We retain the space of dimension $d$ and the orthogonal cyclic action from
\cref{subsec:exact-dimension}.  For every nonempty
$S\subseteq\mathcal N$, put
$$
c_S=\sum_{n\in S}\varphi(n),
\qquad
s_S=rc_S+a\mathbf1_{\{1\in S\}},
\qquad
M_S=\operatorname{lcm}\{n:n\in S\}.
$$
Thus $c_{\mathcal N}=C$, $s_{\mathcal N}=d$, and
$M_{\mathcal N}=M$.  Every nonzero vector $x\in V$ of support $S$
has an orbit of size $M_S$, by the same argument as in
Lemma~\ref{lem:orbit}.  Moreover,
$$
\dim_{\mathbb R}\operatorname{span}_{\mathbb R}
\{\tau^j x:j\in\mathbb Z\}\le c_S.
$$
Indeed, the polynomial $\prod_{n\in S}\Phi_n$, of degree $c_S$,
annihilates $x$ under the action of $\tau$.  These two facts account
for the weight $c_S/M_S$ in the argument below.

\subsection{Stretching subspaces spanned by cyclic orbits}

For $\rho>0$ and $x\ne0$, define
$$
h_\rho(x):=\max\left\{0,\log\frac{\rho}{\|x\|}\right\},
$$
and set $h_\rho(0)=0$.  For a full $\tau$-invariant lattice
$\Lambda\subset V$, define
\begin{equation}
    H_\rho(\Lambda)
    :=
    \sum_{\varnothing\ne S\subseteq\mathcal N}
    \frac{c_S}{M_S}
    \sum_{\substack{x\in\Lambda\\\operatorname{supp}(x)=S}}
    h_\rho(x).
    \label{eq:deterministic-log-sum}
\end{equation}
This sum is finite because only nonzero lattice vectors of length
less than $\rho$ contribute.

We use a version of Minkowski's lemma on successive
minima~\cite{MinkowskiGeometry1910} that preserves the action of
$\mathbb R[\tau]$. Here
$$
\mathbb Q[\tau]\simeq\prod_{n\in\mathcal N}K_n.
$$
For a field or a division algebra, the successive module spans have
equal dimension increments, giving the geometric mean in
\cite[Proposition~4.1]{GargavaSerbanStretching}. For the product of fields
above, the dimension increments $b_i$ can vary. The same orthogonal
stretching argument gives the following statement with exponents $b_i/d$.

\begin{lemma}[Minkowski's lemma]
    \label{lem:minkowski-stretching}
    Let $\Lambda\subset V$ be a full lattice preserved by $\tau$.
    Starting with $W_0=\{0\}$, choose a shortest vector
    $x_i\in\Lambda\setminus W_{i-1}$ and put
    $$
    \mu_i=\|x_i\|,
    \qquad
    W_i=W_{i-1}+\mathbb R[\tau]x_i.
    $$
    Continue until $W_k=V$, and set
    $$
    b_i=\dim W_i-\dim W_{i-1},
    \qquad
    \bar\mu=\prod_{i=1}^k\mu_i^{b_i/d}.
    $$
    There is a positive definite linear map $T_0$ commuting with $\tau$ such that
    $$
    \det T_0=1,
    \qquad
    \lambda_1(T_0\Lambda)=\bar\mu.
    $$
\end{lemma}

Applying \cref{lem:minkowski-stretching} and estimating the contributions
of the cyclic orbits gives a bound on $\log\det T$ in terms of
$H_\rho(\Lambda)$.

\begin{lemma}[Stretching bound for lattices preserved by the cyclic action]
    \label{lem:deterministic-stretching}
    For every full $\tau$-invariant lattice $\Lambda\subset V$ and every
    $\rho>0$, there is a positive definite linear map $T$ commuting with
    $\tau$ such that
    $$
    \lambda_1(T\Lambda)\ge\rho,
    \qquad
    \log\det T\le H_\rho(\Lambda).
    $$
\end{lemma}

\begin{proof}
    Take $x_i,\mu_i,W_i,b_i,\bar\mu$ and $T_0$ from
    \cref{lem:minkowski-stretching}, and set
    $
    T=\frac{\rho}{\bar\mu}T_0.
    $
    Then $T$ is positive definite, commutes with $\tau$, and satisfies
    $\lambda_1(T\Lambda)=\rho$. Since $\sum_i b_i=d$, we have
    $$
    \log\det T
    =\sum_{i=1}^k b_i\log\frac{\rho}{\mu_i}
    \le\sum_{i=1}^k b_i h_\rho(x_i).
    $$
    Put $S_i=\operatorname{supp}(x_i)$. The dimension bound for
    $\mathbb R[\tau]x_i$ gives $b_i\le c_{S_i}$. The selected cyclic orbits
    are pairwise disjoint, since $x_i\notin W_{i-1}$ and $W_{i-1}$ is
    preserved by $\tau$. Each orbit contains $M_{S_i}$ vectors, all of
    support $S_i$ and length $\mu_i$. Its contribution to
    \eqref{eq:deterministic-log-sum} is therefore $c_{S_i}h_\rho(x_i)$.
    Consequently,
    $$
    \log\det T
    \le\sum_{i=1}^k c_{S_i}h_\rho(x_i)
    \le H_\rho(\Lambda).
    $$
\end{proof}

The lemma turns a bound on a logarithmic lattice sum into a packing
bound.  More precisely, the lattice $T\Lambda$ has packing density
at least
\begin{equation}
    \frac{v_d\rho^d}{2^d\covol(\Lambda)}
    \exp\bigl(-H_\rho(\Lambda)\bigr).
    \label{eq:deterministic-density}
\end{equation}
The construction uses a flag of subspaces spanned by cyclic orbits. The dimension increase
at each step may vary with the support and with the preceding spans.

\subsection{The expected value of \texorpdfstring{$H_\rho$}{H rho}}

Use the normalized Haar ensemble from
\eqref{eq:normalized-haar-lattice},
$$
\Lambda_g^*=\lambda T_y gL,
\qquad
\covol(\Lambda_g^*)=v_d.
$$
For $t>0$ and nonempty $S\subseteq\mathcal N$, write
$$
A_S(t):=
\mathbb E_g\,
\#\{x\in\Lambda_g^*:
\operatorname{supp}(x)=S,\ \|x\|<t\}.
$$
The Siegel formula for vectors of each prescribed support in Lemma~\ref{lem:mean} gives
$$
A_S(t)=A_S(\rho)\left(\frac{t}{\rho}\right)^{s_S},
\qquad
A_{\mathcal N}(t)=t^d.
$$
Although $h_\rho$ is unbounded near the origin, the mean-value formula
extends to it by applying the formula to $\min\{h_\rho,K\}$ and
letting $K\to\infty$ by monotone convergence.  Equivalently, using
$$
h_\rho(x)=\int_0^\rho
\mathbf1_{\{\|x\|<t\}}\,\frac{dt}{t}
\qquad(x\ne0),
$$
Tonelli's theorem and the preceding scaling identity give
$$
\mathbb E_g
\sum_{\substack{x\in\Lambda_g^*\\\operatorname{supp}(x)=S}}
h_\rho(x)
=\int_0^\rho A_S(t)\,\frac{dt}{t}
=\frac{A_S(\rho)}{s_S}.
$$
It follows that
\begin{equation}
    \begin{aligned}
        \mathbb E_g H_\rho(\Lambda_g^*)
    &={}
    \frac{C\rho^d}{Md}
    +\sum_{\varnothing\ne S\subsetneq\mathcal N}
    \frac{c_S}{M_S s_S}A_S(\rho).
    \end{aligned}
    \label{eq:deterministic-log-mean}
\end{equation}

The gluing estimates control the second term.  Whenever $\rho$
satisfies \eqref{eq:rho-range}, Corollary~\ref{cor:proper-supports} gives
$$
\sum_{\varnothing\ne S\subsetneq\mathcal N}A_S(\rho)
\le e^{-cr}
$$
for some $c>0$.  Since $s_S\ge rc_S$ and $M_S\ge1$, we obtain
\begin{equation}
    \mathbb E_g H_\rho(\Lambda_g^*)
    \le\frac{C\rho^d}{Md}+\frac{e^{-cr}}r.
    \label{eq:deterministic-mean-bound}
\end{equation}
Thus vectors of proper support make a negligible contribution to
$\mathbb E_g H_\rho(\Lambda_g^*)$.

\subsection{The packing bound in every dimension}

\begin{proof}[Proof of \cref{thm:deterministic-packing}]
    Choose $\rho>0$ by
    $
    \rho^d=\frac{dM}{C}.
    $
    Admissibility gives $C\le d$, $r\le d$, and $\log M=O(\log d)$,
    so
    $$
    |\log\rho|=O\left(\frac{\log d}{d}\right)
    =O\left(d^{-1}+\sqrt{\frac{\log d}{r}}\right).
    $$
    Hence \eqref{eq:deterministic-mean-bound} applies.  Set
    $\epsilon_d=e^{-cr}/r$, which tends to zero.  There is a choice of
    $g$ for which
    $$
    H_\rho(\Lambda_g^*)\le1+\epsilon_d.
    $$
    Apply Lemma~\ref{lem:deterministic-stretching} to this lattice.
    By \eqref{eq:deterministic-density} and
    $\covol(\Lambda_g^*)=v_d$, the resulting lattice has packing density
    at least
    $$
    \frac{\rho^d}{2^d}\exp(-1-\epsilon_d)
    =e^{-\epsilon_d}\frac{dM}{eC}\,2^{-d}.
    $$
    This proves the first assertion.  Applying
    \cref{thm:arithmetic-realization}, which supplies admissible data
    with $M/C\gg\log\log d$ for every sufficiently large integer $d$,
    proves the second assertion.
\end{proof}

The factor $M/C$ comes from comparing the number of vectors in a full
cyclic orbit with the dimension of its span.  The additional factor
$d$ comes from the radial integral of the logarithmic weight.  At the
chosen radius, the expected value of $H_\rho$ is $1+o(1)$,
so a bounded increase in covolume suffices to obtain the claimed
packing density.

\begin{remark}[Direct orbit counting]
For a $\tau$-invariant ball of normalized volume $V$ (Euclidean volume
$Vv_d$), Siegel's formula gives an expected count $V$ of vectors with full support, and
this count is a multiple of $M$. The estimate for vectors of proper support then
shows that a lattice-free ball exists at $V=M-o(1)$, provided
$Me^{-cr}=o(1)$. This yields density $(M-o(1))2^{-d}$.  In the present
construction $M=d^{o(1)}$, so direct orbit counting alone does not give
$d\log\log d\,2^{-d}$.  
\end{remark}
%

\section{Mean-value formula for vectors with prescribed support}
\label{app:mean-value-gluing} 

We prove \cref{lem:mean}. We begin with the mean value theorem over a
number field. For a cyclotomic field and the standard lattice, this is
\cite[Theorem~2.1]{AbuyaGargavaZhao}. The corresponding statement for a
general number field follows from
\cite[Theorem~1.2 and the subsequent discussion]{KimAdelicRogers}.
We state the version for an arbitrary full lattice.

\begin{theorem}[Siegel mean value theorem over a number field]
Let $K$ be a number field, let $K_\infty=K\otimes_{\mathbb Q}\mathbb R$,
and let $r\ge2$. Let $\mathcal O_K$ denote the ring of integers of $K$.
Put $G_K=\operatorname{SL}_r(K_\infty)$. Let
$\Lambda\subseteq K^r$ be a full $\mathbb Z$-lattice, and set
$$
\Gamma_\Lambda
=\{\gamma\in\operatorname{SL}_r(K)\mid\gamma\Lambda=\Lambda\}.
$$
The group $\Gamma_\Lambda$ is arithmetic, so $G_K/\Gamma_\Lambda$ has
finite invariant volume. Normalize this measure to have total mass
one. For every nonnegative bounded compactly supported Borel function
$f$ on $K_\infty^r$,
\begin{equation}
\int_{G_K/\Gamma_\Lambda}
\sum_{v\in g\Lambda\setminus\{0\}}f(v)\,d\mu(g)
=
\frac{1}{\covol(\Lambda)}
\int_{K_\infty^r}f(x)\,dx.
\label{eq:single-field-mean}
\end{equation}
\end{theorem}

For $\Lambda=\mathcal O_K^r$, \cref{eq:single-field-mean} is the
standard form of the theorem. The general lattice statement is its
finite level form.

We now pass to a product of number fields. Let
$K_1,\ldots,K_m$ be number fields and let $r_i\ge2$. Set
$$
V_i=(K_i\otimes_{\mathbb Q}\mathbb R)^{r_i},
\qquad
V=\bigoplus_{i=1}^m V_i,
\qquad
G=\prod_{i=1}^m\operatorname{SL}_{r_i}
       (K_i\otimes_{\mathbb Q}\mathbb R).
$$
Use the product Lebesgue measure on $V$ and the restricted measures on
the subspaces $V_S$.
Choose full lattices $\Lambda_i\subseteq K_i^{r_i}$ and a finite
index sublattice
$$
\Lambda\subseteq\Lambda_1\oplus\cdots\oplus\Lambda_m.
$$
For $S\subseteq\{1,\ldots,m\}$, write
$$
V_S=\bigoplus_{i\in S}V_i,
\qquad
\Lambda_S=\Lambda\cap V_S.
$$
The arithmetic group is
$$
\Gamma=\{\gamma\in\prod_{i=1}^m\operatorname{SL}_{r_i}(K_i)\mid
                         \gamma\Lambda=\Lambda\}.
$$
It fixes $\Lambda$ and is an arithmetic subgroup of the rational
algebraic group
$$
\mathbf G=\prod_{i=1}^m
\operatorname{Res}_{K_i/\mathbb Q}\operatorname{SL}_{r_i}.
$$
Since $\mathbf G$ is semisimple, the Borel--Harish-Chandra theorem
\cite[Thm.~12.3(c)]{BorelHarishChandra} gives finite invariant volume
for $G/\Gamma$. We normalize this measure to a probability measure.

For $v=(v_1,\ldots,v_m)\in\Lambda$, define
$$
\operatorname{supp}(v)=\{i:v_i\ne0\}.
$$
For a compactly supported Borel function $F$ on $V$, define
$$
\nu_\Lambda(F)
=\int_{G/\Gamma}\sum_{v\in\Lambda}F(gv)\,d\mu(g).
$$
This measure is locally finite and invariant under $G$. For each nonempty
$S\subseteq\{1,\ldots,m\}$, the vectors of support $S$ lie in the
$G$ orbit whose components are nonzero in every real or complex factor
of $K_i\otimes_{\mathbb Q}\mathbb R$ for $i\in S$. This orbit has full
Lebesgue measure in $V_S$. It follows that there are constants
$c_S\ge0$ such that, for every nonnegative bounded compactly supported
Borel function $F$ on $V$,
\begin{equation}
\int_{G/\Gamma}\sum_{v\in\Lambda}F(gv)\,d\mu(g)
=F(0)+\sum_{\varnothing\ne S\subseteq\{1,\ldots,m\}}
c_S\int_{V_S}F(x)\,dx.
\label{eq:product-mean-constants}
\end{equation}

For each $S$, take $F$ to be the indicator of a product of balls in
$V_S$, with radii $R_i=\alpha R$ for $i\in S$, extended by zero to $V$.
Substituting this function into the mean value formula, dividing by
$R^{\dim_{\mathbb R}V_S}$, and letting $R\to\infty$ isolates $c_S$, which can be calculated as $\covol(\Lambda_{S})^{-1}$.
The terms from proper subsets of $S$ vanish in this limit and this way we can prove the mean value theorem.

Index the factors by $n\in\mathcal N$ and take $K_i=K_n$, $r_i=r_n$,
and $\Lambda=L$ from \cref{subsec:exact-dimension}. Then $\Lambda_S=L_S$
and the formula
for vectors of exact support $S$ is
$$
\int_{G/\Gamma}\sum_{v\in L_S^\circ}f(gv)\,d\mu(g)
=\frac{1}{\covol(L_S)}\int_{V_S}f(u)\,du.
$$
This proves \cref{lem:mean}.

\clearpage
\bibliographystyle{alpha}
\bibliography{references}

@misc{AbuyaGargavaZhao,
  author       = {Abuya, E. B. and Gargava, N. and Zhao, Y.},
  title        = {Stochastically evolving ellipsoids with symmetries},
  year         = {2026},
  note         = {arXiv:2606.05105v2}
}

@article{Apostol,
  author       = {Apostol, T. M.},
  title        = {Resultants of cyclotomic polynomials},
  journal      = {Proceedings of the American Mathematical Society},
  volume       = {24},
  year         = {1970},
  pages        = {457--462}
}

@article{Klartag,
  author  = {Klartag, Boaz},
  title   = {Lattice packing of spheres in high dimensions using a stochastically evolving ellipsoid},
  journal = {Inventiones Mathematicae},
  volume  = {244},
  number  = {3},
  year    = {2026},
  pages   = {1251--1279},
  doi     = {10.1007/s00222-026-01412-w}
}

@article{CaputoMenzTetali2015,
  author       = {Caputo, P. and Menz, G. and Tetali, P.},
  title        = {Approximate tensorization of entropy at high temperature},
  journal      = {Ann. Fac. Sci. Toulouse Math.},
  volume       = {24},
  number       = {4},
  year         = {2015},
  pages        = {691--716}
}

@book{VershyninHDP,
  author       = {Vershynin, R.},
  title        = {High-Dimensional Probability: An Introduction with Applications in Data Science},
  edition      = {2},
  publisher    = {Cambridge University Press},
  address      = {Cambridge},
  year         = {2026}
}

@article{KimAdelicRogers,
  author       = {Kim, S.},
  title        = {Adelic {Rogers} integral formula},
  journal      = {Journal of the London Mathematical Society},
  series       = {2},
  volume       = {109},
  number       = {1},
  year         = {2024},
  pages        = {e12830}
}

@article{RogersExistence1947,
  author  = {Rogers, C. A.},
  title   = {Existence Theorems in the Geometry of Numbers},
  journal = {The Annals of Mathematics},
  volume  = {48},
  number  = {4},
  pages   = {994--1002},
  year    = {1947},
  doi     = {10.2307/1969390}
}

@book{MinkowskiGeometry1910,
  author    = {Minkowski, Hermann},
  title     = {Geometrie der Zahlen},
  publisher = {B. G. Teubner},
  address   = {Leipzig},
  year      = {1910}
}

@article{BorelHarishChandra,
  author  = {Borel, Armand and Harish-Chandra},
  title   = {Arithmetic Subgroups of Algebraic Groups},
  journal = {Annals of Mathematics},
  volume  = {75},
  number  = {3},
  year    = {1962},
  pages   = {485--535}
}

@article{MahatabSampath,
  author  = {Kamalakshya Mahatab and Kannappan Sampath},
  title   = {Chinese Remainder Theorem for Cyclotomic Polynomials in {$\mathbf{Z}[X]$}},
  journal = {Journal of Algebra},
  volume  = {435},
  pages   = {223--262},
  year    = {2015},
  doi     = {10.1016/j.jalgebra.2015.04.006}
}

@book{ConwaySloane,
  author    = {John H. Conway and Neil J. A. Sloane},
  title     = {Sphere Packings, Lattices and Groups},
  edition   = {3},
  publisher = {Springer},
  year      = {1999}
}

@article{Venkatesh,
  author  = {Akshay Venkatesh},
  title   = {A Note on Sphere Packings in High Dimension},
  journal = {International Mathematics Research Notices},
  volume  = {2013},
  number  = {7},
  pages   = {1628--1642},
  year    = {2013},
  doi     = {10.1093/imrn/rns096}
}

@misc{GargavaViazovska,
  author       = {Nihar Gargava and Maryna Viazovska},
  title        = {Mean Value for Random Ideal Lattices},
  year         = {2024},
  note         = {arXiv:2411.14973}
}

@misc{GargavaSerbanViazovska,
  author       = {Nihar Gargava and Vlad Serban and Maryna Viazovska},
  title        = {Moments of the Number of Points in a Bounded Set for Number Field Lattices},
  year         = {2023},
  note         = {arXiv:2308.15275}
}

@misc{GargavaModuleShortest,
  author       = {Nihar Gargava and Vlad Serban and Maryna Viazovska and Ilaria Viglino},
  title        = {Module Lattices and Their Shortest Vectors},
  year         = {2025},
  note         = {arXiv:2510.12893}
}

@misc{GargavaSerbanStretching,
  author       = {Nihar Gargava and Vlad Serban},
  title        = {Dense Packings via Lifts of Codes to Division Rings},
  year         = {2022},
  note         = {arXiv:2111.03684v2}
}

\end{document}